\documentclass[a4paper,twoside,11pt]{amsbook}
\usepackage{amsfonts, amsbsy, amsmath, amsthm, amssymb, latexsym, verbatim, enumerate}
\usepackage{mathrsfs}
\usepackage{pdfsync}

\usepackage{bm}

\usepackage[pdftex]{color,graphicx}

\makeatletter
\newcommand{\imod}[1]{\allowbreak\mkern4mu({\operator@font mod}\,\,#1)}
\makeatother

\usepackage{tikz}
\usetikzlibrary{calc}
\usetikzlibrary{shapes.geometric}
\usetikzlibrary{trees}

\numberwithin{equation}{section}

\usepackage[official]{eurosym}
\usepackage{array,float,graphicx,epic,eepic}
\usepackage{amssymb, amsmath, amsthm}
\input xy
\xyoption{all}
\def \la {\langle}
\def \ra {\rangle}

\renewcommand{\a}{\alpha}
\renewcommand{\b}{\beta}

\newcommand{\e}{\epsilon}
 
\renewcommand{\l}{\lambda} \renewcommand{\O}{\Omega}

 \renewcommand{\to}{\rightarrow}
 \newcommand{\s}{\sigma}

\newcommand{\leqs}{\leqslant}

\newcommand{\geqs}{\geqslant}
 
\newcommand{\what}{\widehat} 
 \newcommand{\vs}{\vspace{3mm}}

\usepackage{chngcntr}
\counterwithin{table}{section}
\counterwithin{figure}{section}

\newtheorem{thm}{Theorem}[section] 
\newtheorem{prop}[thm]{Proposition} 
\newtheorem{lem}[thm]{Lemma}
\newtheorem{cor}[thm]{Corollary} 
\newtheorem{con}[thm]{Conjecture}

\newtheorem*{quest}{Question}

\theoremstyle{definition}
\newtheorem{defn}[thm]{Definition}

\newtheorem{rem}[thm]{Remark}

\newtheorem{ex}[thm]{Example}
\newtheorem{exs}[thm]{Examples}

\begin{document}

\mainmatter

\title{Simple groups, fixed point ratios \\ and applications}

\author[T.C. Burness]{Timothy C. Burness}
 
 \address{School of Mathematics \\ University of Bristol \\ Bristol BS8 1TW \\ United Kingdom}

\begin{abstract}
The study of fixed point ratios is a classical topic in permutation group theory, with a long history stretching back to the origins of the subject in the 19th century. Fixed point ratios arise naturally in many different contexts, finding a wide range of applications. In this survey article we focus on fixed point ratios for simple groups of Lie type, highlighting 
some of the main results, applications and related problems.
\end{abstract}
 
\maketitle

\section{Introduction}\label{s:intro}

The study of fixed point ratios is a classical topic in permutation group theory, with a long history stretching all the way back to the early days of group theory in the 19th century. The concept arises naturally in many different contexts, finding a wide range of interesting (and often surprising) applications. One of the main aims of this survey article is to highlight some of these applications. For instance, we will explain how fixed point ratios play a key role in the study of some remarkable generation properties of finite groups. We will also see how probabilistic methods, based on fixed point ratio estimates, have revolutionised the search for small bases of primitive permutation groups. In a completely different direction, we will also describe how bounds on fixed point ratios can be used to investigate the structure of monodromy groups of coverings of the Riemann sphere. 

In this introductory section we start by recalling some basic properties of fixed point ratios and we present several standard examples that will be useful later. We also highlight connections to some classical notions in permutation group theory, such as minimal degree, fixity and derangements. To whet the appetite of the reader, we close the introduction by presenting three very different group-theoretic problems. It is interesting to note that none of these problems have an obvious connection to fixed point ratios, but we will show later that recent advances in our understanding of fixed point ratios (in particular, recent results for (almost) simple groups of Lie type) play an absolutely essential role in their solution. 

Some of the  main theorems on fixed point ratios will be highlighted in Section \ref{s:simple}, where we focus on the simple groups of Lie type. Finally, in Sections \ref{s:gen}, \ref{s:mono} and \ref{s:bases} we will discuss the three motivating problems mentioned above. Here we will explain the connection to fixed point ratios and we will sketch some of the main ideas. In particular, we will see how some of the results presented in Section \ref{s:simple} play a key role, and we will report on more recent developments and open problems. 

Finally, we have also included an extensive bibliography, which we hope will serve as a useful guide for further reading.

\vs

\noindent \emph{Acknowledgments.} This article is based on a lecture series I gave at the workshop \emph{Some problems in the theory of simple groups} at the Centre Interfacultaire Bernoulli at the \'{E}cole Polytechnique F\'{e}d\'{e}rale de Lausanne in September 2016. This workshop was part of the semester  programme \emph{Local representation theory and simple groups} at the CIB. It is a pleasure to thank the organisers of this programme for inviting me to participate and for their generous hospitality. I would also like to thank Gunter Malle and Donna Testerman for their helpful comments on an earlier version of this article.

\subsection{Preliminaries}

We start with some preliminary definitions. Let $G \leqs {\rm Sym}(\O)$ be a permutation group on a finite set $\O$. For $\a \in \O$ we will write $G_{\a} = \{x \in G \,:\, \a^x = \a\}$ for the stabiliser of $\a$ in $G$. Similarly, the set of fixed points of $x \in G$ will be denoted by  
$$C_{\O}(x) = \{\a \in \O \,:\, \a^x = \a\}.$$

\begin{defn}
The \emph{fixed point ratio} of $x \in G$, denoted by ${\rm fpr}(x,\O) = {\rm fpr}(x)$, is the proportion of points in $\O$ fixed by $x$, i.e. 
$${\rm fpr}(x) = \frac{|C_{\O}(x)|}{|\O|}.$$
\end{defn}

Notice that ${\rm fpr}(x)$ is the \emph{probability} that a randomly chosen element of $\O$ is fixed by $x$ (with respect to the uniform distribution on $\O$). This viewpoint is often useful for applications. Indeed, in recent years probabilistic methods have been used to solve many interesting problems in finite group theory. Typically, the aim is to establish an existence result through a probabilistic approach (rather than an explicit construction, for example) -- this has been a standard technique in combinatorics, number theory and other areas for many years. As we will see later, bounds on fixed point ratios play a central role in several applications of this flavour.

It is also worth noting that a fixed point ratio is a special type of \emph{character ratio}; if $\pi:G \to \mathbb{C}$ is the corresponding permutation character, then 
$${\rm fpr}(x) = \frac{\pi(x)}{\pi(1)}.$$

The following lemma records some basic properties. 

\begin{lem}\label{l:basic}
Let $G$ be a permutation group on a finite set $\O$ and let $x$ be an element of $G$. 
\begin{itemize}\addtolength{\itemsep}{0.2\baselineskip}
\item[{\rm (i)}] ${\rm fpr}(x)={\rm fpr}(y)$ for all $y \in x^G$.
\item[{\rm (ii)}] ${\rm fpr}(x) \leqs {\rm fpr}(x^m)$ for all $m \in \mathbb{Z}$.
\item[{\rm (iii)}] If $G$ is transitive with point stabiliser $H$, then
$${\rm fpr}(x) = \frac{|x^G \cap H|}{|x^G|}.$$
\item[{\rm (iv)}] If the derived subgroup $G'$ is transitive, then there is a non-linear irreducible constituent $\chi$ of the permutation character such that 
$${\rm fpr}(x) \leqs \frac{1+|\chi(x)|}{1+\chi(1)}.$$
\end{itemize} 
\end{lem}

\begin{proof}
Parts (i) and (ii) are trivial. 
\begin{itemize}\addtolength{\itemsep}{0.2\baselineskip}
\item[(iii)] For $\b \in \O$, $g \in G$ define $(\b,g) = 1$ if $\b^g=\b$, otherwise $(\b,g)=0$. Then
\begin{align*}
|x^G|\,|C_{\O}(x)| = \sum_{g \in x^G} |C_{\O}(g)| & =  \sum_{g \in x^G} \left(\sum_{\b \in \O}(\b,g)\right) \\
 & = \sum_{\b \in \O}\left(\sum_{g \in x^G}(\b,g)\right) = \sum_{\b \in \O}|x^G \cap G_{\b}|,
 \end{align*}
which is equal to $|\O|\, |x^G \cap H|$ by the transitivity of $G$. 
\item[(iv)] Set $f={\rm fpr}(x)$ and write $\pi = 1+\chi_1+ \cdots + \chi_t$ with $\chi_i \in {\rm Irr}(G)$. Note that the transitivity of $G'$ implies that each $\chi_i$ is non-linear. If $1+|\chi_i(x)|<f(1+\chi_i(1))$ for all $i$ then
\begin{align*}
f|\O| = 1+\sum_{i}\chi_i(x) & \leqs 1+\sum_{i}|\chi_i(x)| \\
& = \sum_{i}(1+|\chi_i(x)|)-(t-1) \\
& < f\left(\sum_{i}(1+\chi_i(1))\right)-(t-1) \\
& = f|\O| - (1-f)(t-1),
\end{align*}
which is a contradiction. \qedhere 
\end{itemize}
\end{proof}

The formula in part (iii) of the previous lemma is a key tool for computing fixed point ratios for transitive groups. Indeed, it essentially reduces the problem to determining the fusion of $H$-classes in $G$, which may be more tractable.

\begin{ex}\label{ex:sn2}
\emph{${\rm Sym}(n)$ on $2$-sets.}

\vs

\noindent Let $G = {\rm Sym}(n)$ be the symmetric group of degree $n \geqs 5$, let $x = (1,2,3) \in G$ and let $\O$ be the set of $2$-element subsets of $\{1, \ldots, n\}$. Then $|\O| = \binom{n}{2}$ and the action of $G$ is transitive, with point stabiliser $H = {\rm Sym}(n-2) \times {\rm Sym}(2)$. We compute ${\rm fpr}(x)$ in three different ways:

\vs

\noindent a. \emph{Direct calculation.} We have $C_{\O}(x) = \{\{a,b\} \,:\, a,b \in \{4, \ldots, n\}\}$, so $|C_{\O}(x)| = \binom{n-3}{2}$ and thus
\begin{equation}\label{e:fprxx}
{\rm fpr}(x) = \frac{(n-3)(n-4)}{n(n-1)}.
\end{equation}

\vs

\noindent b. \emph{Permutation character.} Let $\pi$ be the permutation character. The action of $G$ on $\O$ has rank $3$ (that is, $H$ has three orbits on $\O$) and by \emph{Young's Rule} we have 
$$\pi = 1+\chi^{(n-1,1)}+\chi^{(n-2,2)},$$ 
where $\chi^{\lambda}$ is the character of the irreducible \emph{Specht module} $S^{\l}$ corresponding to the partition $\l$ of $n$ (see \cite[Section 14]{James}, for example). By applying the \emph{Hook Formula} for dimensions and the \emph{Murnaghan-Nakayama Rule} for character values (see \cite[Sections 20 and 21]{James}), we calculate that  
$$\chi^{(n-1,1)}(x) = \chi^{(n-4,1)}(1) = n-4$$
and
$$\chi^{(n-2,2)}(x) = \chi^{(n-5,2)}(1) = (n-3)(n-6)/2$$
if $n \geqs 7$ (one can check that $\chi^{(3,2)}(x)=-1$ and $\chi^{(4,2)}(x)=0$), so 
$${\rm fpr}(x) = \frac{1+(n-4)+(n-3)(n-6)/2}{n(n-1)/2} = \frac{(n-3)(n-4)}{n(n-1)}.$$

\vs

\noindent c. \emph{Conjugacy classes.} All the $3$-cycles in $G$ are conjugate, so $x^G \cap H$ is the set of $3$-cycles in $H$. This gives
$$|x^G \cap H| = |x^H| = 2\binom{n-2}{3}, \;\; |x^G| = 2\binom{n}{3}$$ 
and thus Lemma \ref{l:basic}(iii) implies that \eqref{e:fprxx} holds.
\end{ex}

\begin{ex}\label{ex:gl}
\emph{${\rm GL}_{n}(q)$ on vectors.}

\vs

\noindent Consider the action of $G = {\rm GL}_{n}(q)$ on its natural module $V = \mathbb{F}_{q}^n$. For $x \in G$ we have ${\rm fpr}(x) = q^{d-n}$, where $d$ is the dimension of the $1$-eigenspace of $x$ on $V$. 
\end{ex}

\begin{ex}\label{ex:pgl}
\emph{${\rm PGL}_{n}(q)$ on $1$-spaces.}

\vs

\noindent Similarly, we can consider the transitive action of $G={\rm PGL}_{n}(q) = {\rm GL}_{n}(q)/Z$ on the set of $1$-dimensional subspaces of $V$. Suppose $q$ is odd and set $x = \hat{x}Z \in G$, where $\hat{x} \in {\rm GL}_{n}(q)$ is the block-diagonal matrix $[-I_1, I_{n-1}]$ with respect to a basis $\{e_1, \ldots, e_n\}$ for $V$ (here, and elsewhere, we use $I_m$ to denote the $m \times m$ identity matrix). Then $x$ fixes $\la e_1 \ra$ and every $1$-space in $\la e_2, \ldots, e_n\ra$, and no others, so 
$${\rm fpr}(x) = \frac{1+\frac{q^{n-1}-1}{q-1}}{\frac{q^n-1}{q-1}} = \frac{q^{n-1}+q-2}{q^n-1} \sim \frac{1}{q}.$$
Alternatively, note that a point stabiliser $H = q^{n-1}{:}({\rm GL}_{1}(q) \times {\rm GL}_{n-1}(q))/Z$ is a maximal parabolic subgroup of $G$ and one checks that $x^G \cap H$ is a union of two $H$-classes. More precisely, 
$$|x^G \cap H| = q^{n-1}+q\cdot \frac{|{\rm GL}_{n-1}(q)|}{|{\rm GL}_{n-2}(q)||{\rm GL}_{1}(q)|},\;\; |x^G| = \frac{|{\rm GL}_{n}(q)|}{|{\rm GL}_{n-1}(q)||{\rm GL}_{1}(q)|}$$
which provides another way to compute ${\rm fpr}(x)$ via Lemma \ref{l:basic}(iii).
\end{ex}

\subsection{Problems}

It is natural to consider the following problems, either in the context of a specific permutation group, or more typically for an interesting family of permutation groups, such as primitive groups and almost simple groups. 

\vs

\begin{itemize}\addtolength{\itemsep}{0.4\baselineskip}
\item[1.] Given a permutation group $G$ and $x \in G$, compute ${\rm fpr}(x)$.

\item[2.] Obtain upper and lower bounds on ${\rm fpr}(x)$ (in terms of parameters depending on $G$ and $x$).

\item[3.] Compute (or bound) the minimal and maximal fixed point ratios $\min\{ {\rm fpr}(x) \,:\, x \in G\}$ and $\max\{ {\rm fpr}(x) \,:\, 1 \ne x \in G\}$.

\item[4.] We can also consider ``local" versions. For example, given a (normal) subset $S \subseteq G \setminus \{1\}$, compute (or bound) $\min\{ {\rm fpr}(x) \,:\, x \in S\}$ and $\max\{ {\rm fpr}(x) \,:\, x \in S\}$. 

\vs

\noindent For instance, we may be interested in the case where $S$ is the set of elements of prime order in $G$, or the set of involutions, etc. Note that  
$\max\{ {\rm fpr}(x) \,:\, 1 \ne x \in G\}$ and $\max\{ {\rm fpr}(x) \,:\, \mbox{$x \in G$, $|x|$ prime}\}$ are equal by Lemma \ref{l:basic}(ii).
\end{itemize}

As we will see later, bounds on fixed point ratios (in particular, \emph{upper} bounds) are often sufficient for the applications we have in mind. 

\vs

The above problems are closely related to some classical notions in permutation group theory. To see the connection, let us fix a permutation group $G \leqs {\rm Sym}(\O)$ of degree $n$. 

\vs

\noindent a. \emph{Minimal degree:} The minimal degree $\mu(G)$ of $G$ is defined to be the smallest number of points moved by any non-identity element, i.e.
\[\mu(G) = \min_{1 \ne x \in G} \left(n - |C_{\O}(x)|\right) = n \left(1 - \max_{1 \ne x \in G} {\rm fpr}(x)\right).\]
For example, $\mu({\rm Sym}(n))=2$ and $\mu({\rm Alt}(n))=3$. This is a classical invariant studied by Jordan, Bochert, Manning and others (see Section \ref{ss:md}).

\vs

\noindent b. \emph{Fixity:} Similarly, the largest number of fixed points of a non-identity element is called the \emph{fixity} of $G$, denoted by
$$f(G) = n\left(\max_{1 \ne x \in G} {\rm fpr}(x)\right) = n-\mu(G).$$ 
In addition, $\max_{1 \ne x \in G} {\rm fpr}(x)$ is sometimes referred to as the \emph{fixity ratio}. This has been studied by Liebeck, Saxl, Shalev and others. 

If we take $S$ to be the set of involutions in $G$, then $n\left(\max_{x \in S} {\rm fpr}(x)\right)$ is the \emph{involution fixity} of $G$. This concept was studied by Bender in the early 1970s, who classified the transitive groups with involution fixity $1$ (for example, the $3$-transitive action of ${\rm PSL}_{2}(2^m)$ on the projective line has this property). See \cite{BT,Cov,LS15} for more recent results in the context of almost simple primitive groups.

\vs

\noindent c. \emph{Derangements:} An element $x \in G$ is a \emph{derangement} if ${\rm fpr}(x)=0$, so 
$$\min_{1 \ne x \in G} {\rm fpr}(x)=0 \iff \mbox{$G$ contains a derangement}.$$ 
The existence and abundance of derangements has been intensively studied for many years, finding a wide range of applications. We refer the reader to \cite[Chapter 1]{BG_book}, and the references therein.

\subsection{Applications}

The above problems have an intrinsic interest in their own right, but much of the motivation for studying fixed point ratios stems from the wide range of applications. In order to motivate some of these applications, we close this introduction by presenting three very different problems involving simple groups where fixed point ratios play a key role. We will return to these problems in Sections \ref{s:gen}, \ref{s:mono} and \ref{s:bases}.

\vs

\noindent 1. \emph{Generating graphs.} 

\vs

\noindent Let $G$ be a finite group. The \emph{generating graph} $\Gamma(G)$ is a graph on the non-identity elements of $G$ so that two vertices $x,y$ are joined by an edge if and only if $G = \la x,y\ra$.

\vs

\textsc{Problem A.} \emph{Let $G$ be a nonabelian finite simple group. Prove that $\Gamma(G)$ is a connected graph with diameter $2$.}

\vs

\noindent 2. \emph{Monodromy groups.} 

\vs

\noindent  Let $g$ be a non-negative integer and let $\mathcal{E}(g)$ be the set of nonabelian non-alternating composition factors of monodromy groups of branched coverings $f:X \to \mathbb{P}^1(\mathbb{C})$ of the Riemann sphere, where $X$ is a compact connected Riemann surface of genus $g$.

\vs

\textsc{Problem B.} \emph{Prove that $\mathcal{E}(g)$ is finite.}

\vs

\noindent 3. \emph{Bases.} 

\vs

\noindent Let $G \leqs {\rm Sym}(\O)$ be a permutation group. A subset $B \subseteq \O$ is a \emph{base} for $G$ if the pointwise stabiliser of $B$ in $G$ is trivial. The \emph{base size} $b(G)$ of $G$ is the minimal size of a base for $G$. 

\vs

\textsc{Problem C.} \emph{Let $G \leqs {\rm Sym}(\O)$ be a transitive nonabelian finite simple group with point stabiliser $H$ satisfying the following conditions:
\vspace{1mm}
\begin{itemize}\addtolength{\itemsep}{0.4\baselineskip}
\item[{\rm (i)}] If $G = {\rm Alt}(m)$, then $H$ acts primitively on $\{1, \ldots, m\}$.
\item[{\rm (ii)}] If $G$ is a classical group, then $H$ acts irreducibly on the natural module.
\end{itemize}
\vspace{1mm}
Prove that $b(G) \leqs 7$, with equality if and only if $G$ is the Mathieu group ${\rm M}_{24}$ in its natural action on $24$ points.}

\section{Simple groups}\label{s:simple}

In this section we focus on fixed point ratios for primitive simple groups of Lie type. We start with a brief discussion of primitivity in Section \ref{ss:prim}, before turning our attention to the connection between fixed point ratios and the classical notion of minimal degree. In Section \ref{ss:fprr} we introduce a theorem of Liebeck and Saxl (see Theorem \ref{t:liesax}), which provides an essentially best possible upper bound on fixed point ratios for simple groups of Lie type. For the remainder of the section, we look at ways in which this theorem can be strengthened in special cases of interest. For example, we will explain how much stronger bounds have been established for so-called \emph{non-subspace} actions of classical groups -- later we will see that these improved fixed point ratio estimates are essential for the applications we have in mind. 

\vs

Finally, a word or two on notation. For the remainder of this article we will adopt the notation for simple groups used by Kleidman and Liebeck (see \cite[Section 5.1]{KL}). Notice that this differs slightly from the notation in the Atlas \cite{atlas}. For instance, we will write ${\rm P\O}_{n}^{\e}(q)$ for a simple orthogonal group (where $\e=\pm$ when $n$ is even) and ${\rm O}_{n}^{\e}(q)$ is the isometry group of the underlying quadratic form.

\subsection{Primitivity}\label{ss:prim}

Let $G \leqs {\rm Sym}(\O)$ be a permutation group, with orbits $\O_i$, $i \in I$. Then $G$ induces a transitive permutation group $G^{\O_i}$ on each $\O_i$; these are called the \emph{transitive constituents} of $G$. In some sense, $G$ is built from its transitive constituents; indeed, $G$ is a subdirect product of the $G^{\O_i}$ (that is, the corresponding projection maps $G \to G^{\O_i}$ are surjective). For example, $G = \{1, (1,2)(3,4)\}$ has orbits $\O_1 = \{1,2\}$ and $\O_2 = \{3,4\}$ on $\O = \{1,2,3,4\}$, and $G$ is a proper subdirect product of $G^{\O_1} = \{1,(1,2)\}$ and $G^{\O_2} = \{1,(3,4)\}$. For the purposes of studying fixed point ratios, it is natural to assume that $G$ is transitive.

In turn, the transitive constituents themselves may be built from smaller permutation groups in a natural way. This leads us to the notion of  
\emph{primitivity}. This is an important irreducibility condition that allows us to define the \emph{primitive groups}, which are the basic building blocks of all permutation groups. 

\begin{defn}\label{d:prim}
A transitive group $G \leqs {\rm Sym}(\O)$ is \emph{imprimitive} if $\O$ admits a nontrivial $G$-invariant partition, otherwise $G$ is \emph{primitive}.
\end{defn}

Here the trivial partitions are $\{\O\}$ and $\{\{\a\} \,:\, \a \in \O\}$. It is an easy exercise to show that $G$ is primitive if and only if a point stabiliser $H = G_{\a}$ is a maximal subgroup of $G$, which is a useful characterisation. For instance, the action of $G={\rm Sym}(n)$ on the set of $k$-element subsets of $\{1, \ldots, n\}$ is primitive for all $1 \leqs k < n$, $k \ne n/2$ (note that $G$ is imprimitive if $k=n/2$ since $G_{\a}<{\rm Sym}(n/2) \wr {\rm Sym}(2) <G$). Any transitive group of prime degree is primitive and all $2$-transitive groups are primitive.

It turns out that the abstract structure of a finite primitive group $G$ is rather restricted (observe that transitivity alone imposes no structural restrictions whatsoever). For example, the socle of $G$ (denoted ${\rm soc}(G)$) is a direct product of isomorphic simple groups (recall that the \emph{socle} of a group  is the product of its minimal normal subgroups). In fact, we can say much more. The main result is the \emph{O'Nan-Scott theorem} (see \cite[Chapter 4]{DM}, for example), which describes the structure and action of a primitive group in terms of its socle. This is a very powerful tool for studying primitive groups. Indeed, in many situations it can be used to reduce a general problem to a much more specific problem concerning almost simple groups, at which point one can appeal to the \emph{Classification of Finite Simple Groups} (CFSG) and the vast literature on simple groups and their subgroups, conjugacy classes and representations. (Recall that a finite group $G$ is \emph{almost simple} if ${\rm soc}(G) = G_0$ is a nonabelian simple group, so $G_0 \leqs G \leqs {\rm Aut}(G_0)$.)

In view of these observations, in this article we will focus our attention on fixed point ratios for almost simple primitive permutation groups.

\subsection{Minimal degree}\label{ss:md}

Let $G \leqs {\rm Sym}(\O)$ be a primitive permutation group of degree $n$. Recall that  
$$\mu(G) = \min_{1 \ne x \in G} \left(n - |C_{\O}(x)|\right) = n \left(1 - \max_{1 \ne x \in G} {\rm fpr}(x)\right)$$
is the \emph{minimal degree} of $G$. This invariant has been studied since the 19th century. In particular, a classical problem is to find lower bounds on $\mu(G)$ in terms of $n$, assuming $G \ne {\rm Alt}(n), {\rm Sym}(n)$, which is equivalent to finding upper bounds on $\max_{1 \ne x \in G} {\rm fpr}(x)$. We record some results:

\begin{itemize}\addtolength{\itemsep}{0.2\baselineskip}
\item Jordan \cite{Jo71}, 1871: $\mu(G)$ tends to infinity as $n$ tends to infinity. 
In particular, there are only finitely many primitive groups with a given minimal degree bigger than $3$.
\item Bochert \cite{Boc}, 1892: $\mu(G) \geqs n/4-1$ if $G$ is $2$-transitive.
\item Babai \cite{Babai2, Babai1}, 1981/2: $\mu(G) \geqs (\sqrt{n}-1)/2$ (independent of CFSG).
\item Liebeck \& Saxl \cite{LSax}, 1991: $\mu(G) \geqs 2(\sqrt{n}-1)$ (using CFSG).
\end{itemize}

\begin{rem}
The bounds obtained by Babai and Liebeck \& Saxl are essentially best possible. To see this, consider the primitive \emph{product action} of $G = {\rm Sym}(m) \wr {\rm Sym}(2)$ on $n=m^2$ points (with $m \geqs 3$), so
$$(\gamma_1, \gamma_2)^{(x_1,x_2)\pi} = \left\{\begin{array}{ll}
((\gamma_{1})^{x_{1}}, (\gamma_{2})^{x_{2}}) & \mbox{if $\pi=1$} \\
((\gamma_{2})^{x_{2}}, (\gamma_{1})^{x_{1}}) & \mbox{if $\pi=(1,2)$}
\end{array}\right.$$
for all $\gamma_1,\gamma_2 \in \{1, \ldots, m\}$ and $(x_1,x_2)\pi \in G$. One checks that every non-identity element $x \in G$ moves at least $2m$ points, with equality if and only if $x$ is of the form $(y,1)$ or $(1,y)$ in the base group ${\rm Sym}(m)^2$, where $y$ is a transposition. Therefore $\mu(G) = 2\sqrt{n}$.
\end{rem}

The following theorem of Guralnick and Magaard is a simplified version of \cite[Theorem 1]{GM}; it is the best known result on the minimal degree of primitive groups.

\begin{thm}\label{t:gm}
Let $G \leqs {\rm Sym}(\O)$ be a primitive group of degree $n$ with $\mu(G)<n/2$. Then one of the following holds:
\begin{itemize}\addtolength{\itemsep}{0.2\baselineskip}
\item[{\rm (i)}] $G={\rm Sym}(n)$ or ${\rm Alt}(n)$;
\item[{\rm (ii)}] $G \leqs L \wr {\rm Sym}(r)$ acts with its product action on $\O = \Gamma^r$ for some $r \geqs 1$, where $L \leqs {\rm Sym}(\Gamma)$ is an almost simple primitive group with socle $L_0$ and either  
\begin{itemize}\addtolength{\itemsep}{0.2\baselineskip}
\item[{\rm (a)}] $L_0={\rm Alt}(m)$ and $\Gamma$ is the set of $k$-element subsets of $\{1, \ldots, m\}$ for some $k \geqs 1$; or
\item[{\rm (b)}] $L_0=\O_{m}^{\e}(2)$ is an orthogonal group over $\mathbb{F}_{2}$ and $\Gamma$ is a set of $1$-dimensional subspaces of the natural $L_0$-module.
\end{itemize}
\end{itemize}
\end{thm}

By carefully analysing the cases arising in (b), Guralnick and Magaard establish the following striking corollary (see \cite[Corollary 1]{GM}).

\begin{cor}\label{c:gm}
Let $G$ be a finite primitive group and assume that the socle of $G$ is not a product of alternating groups. Then 
$$\max_{1 \ne x \in G}{\rm fpr}(x) \leqs \frac{4}{7}.$$
\end{cor}

\vs

\begin{rem}
\mbox{ }
\begin{itemize}\addtolength{\itemsep}{0.2\baselineskip}
\item[{\rm (i)}] The upper bound in Corollary \ref{c:gm} is best possible. For example, suppose 
$$G = {\rm O}_7(2) \cong {\rm Sp}_{6}(2),  \;\; H=G_{\a} = {\rm O}_{6}^{-}(2)$$ 
and $x \in G$ is a transvection (in other words, $x$ is an involution with Jordan form $[J_2, J_1^4]$ on the natural module for ${\rm Sp}_{6}(2)$, where $J_i$ denotes a standard unipotent Jordan block of size $i$). All the transvections in $H$ (and also in $G$) are conjugate, so 
$$|x^G \cap H| = |x^H| = \frac{|{\rm O}_{6}^{-}(2)|}{2|{\rm Sp}_{4}(2)|} = 36,\;\; |x^G| = \frac{|{\rm Sp}_{6}(2)|}{2^5|{\rm Sp}_{4}(2)|} = 63$$
and thus ${\rm fpr}(x)= 36/63 = 4/7$ (the respective centraliser orders can be read off from \cite[Sections 7 and 8]{ASe}, noting that $x$ is a $b_1$-type involution in both $H$ and $G$).

\item[{\rm (ii)}] Note that the conclusion is false if we allow groups whose  socle is a product of alternating groups. For instance, in Example \ref{ex:sn2} we observed that 
$$\lim_{n \to \infty}\left(\max_{1 \ne x \in G}{\rm fpr}(x)\right) = 1$$ 
for the action of $G = {\rm Sym}(n)$ on $2$-sets.
\end{itemize}
\end{rem}

We refer the reader to \cite{KPS} for results on the minimal degree of arbitrary finite permutation groups and some interesting applications to quantum computing.

\subsection{Fixed point ratios for simple groups}\label{ss:fprr}

In this section we discuss fixed point ratios for almost simple groups of Lie type. With a view towards applications, we are primarily interested in obtaining upper bounds, so it is natural to focus on primitive actions and prime order elements.

We start by recalling a theorem of Liebeck and Saxl \cite[Theorem 1]{LSax}, which is the most general result in this area.

\begin{thm}\label{t:liesax}
Let $G \leqs {\rm Sym}(\O)$ be a transitive almost simple group of Lie type over $\mathbb{F}_{q}$ with socle $G_0$ and point stabiliser $H$. Assume $G_0 \ne {\rm PSL}_{2}(q)$. Then either 
\begin{equation}\label{e:43qq}
\max_{1 \ne x \in G}{\rm fpr}(x) \leqs \frac{4}{3q}
\end{equation}
or $G_0 \in \{{\rm PSL}_{4}(2), {\rm PSp}_{4}(3), {\rm P\O}_{4}^{-}(3)\}$.
\end{thm}

\begin{rem}\label{r:ls}
\mbox{ } 
\begin{itemize}\addtolength{\itemsep}{0.2\baselineskip}
\item[(i)] This is a simplified version of \cite[Theorem 1]{LSax}, which includes the case $G_0 = {\rm PSL}_{2}(q)$ and gives a precise description of the triples $(G,H,x)$ with ${\rm fpr}(x)>4/3q$. 
\item[(ii)] The upper bound is essentially best possible. For instance, in Example \ref{ex:pgl} ($G = {\rm PGL}_{n}(q)$ on $1$-spaces) we observed that there are elements $x \in G$ with ${\rm fpr}(x) \sim 1/q$.
\item[(iii)] Consider the special case $G_0 = {\rm PSL}_{4}(2) \cong {\rm Alt}(8)$ appearing in the statement of the theorem. If $G = G_0.2 = {\rm Sym}(8)$, $|\O|=8$ and $x = (1,2)$, then ${\rm fpr}(x) = 6/8>4/6$.
\end{itemize}
\end{rem}

The proof of Theorem \ref{t:liesax} proceeds by induction, with the ultimate goal of eliminating the existence of a minimal counterexample (minimal with respect to the order of the group). The details of the argument are somewhat complicated by the fact that there are a small number of groups for which the bound in \eqref{e:43qq} is false. To give a flavour of the main ideas, we provide a brief sketch to show that \eqref{e:43qq} holds when $G = {\rm PSL}_{n}(q)$ with $n \geqs 6$. Below we use the notation $P_m$ for the stabiliser in $G$ of an $m$-dimensional subspace of the natural module $V$ for $G$.

\vs

\textsc{Sketch proof of Theorem \ref{t:liesax} ($G={\rm PSL}_{n}(q)$, $n \geqs 6$).} Suppose ${\rm fpr}(x)>4/3q$ for some $1 \ne x \in G$. Set $H=G_{\a}$ and write $x = \hat{x}Z$, where $\hat{x} \in {\rm SL}_{n}(q)$ and $Z=Z({\rm SL}_{n}(q))$. In view of Lemma \ref{l:basic}, we may assume that $H$ is maximal (so $G$ is primitive) and $x$ has prime order $r$, so $x$ is either semisimple (if $r \ne p$) or unipotent (if $r=p$), where $p$ is the characteristic of $\mathbb{F}_{q}$. Replacing $x$ by a suitable conjugate, if necessary, we may assume that $x \in H$. Note that   
\begin{equation}\label{e:43q}
|\O| < \frac{3q|C_G(x)|}{4}
\end{equation}
by Lemma \ref{l:basic}(iii).

\vs

Our first goal is to reduce to the case where $x$ stabilises a nontrivial decomposition $V = V_1 \oplus V_2$. Suppose otherwise. 

If $x$ is semisimple then it acts irreducibly on $V$ and we deduce that $|C_G(x)| \leqs (q^n-1)/(q-1)$. In view of \eqref{e:43q}, this implies that $H=P_1$ (the smallest permutation representation of $G$ has degree $(q^n-1)/(q-1)$), which means that $x$ fixes a $1$-space. This is incompatible with the irreducibility of $x$. Similarly, if $x$ is unipotent then it must be regular (i.e. it has Jordan form $[J_n]$ on $V$) and by considering $|C_G(x)|$ we again deduce that $H=P_1$. But a regular unipotent element fixes a unique $1$-dimensional subspace of $V$, so $|C_{\O}(x)| = 1$ and once again we have reached a contradiction.

\vs

Let $V = V_1 \oplus V_2$ be a nontrivial decomposition fixed by $x$ with $1 \leqs a_1 \leqs a_2$, where $a_i = \dim V_i$. Assume $a_1$ is minimal. We claim that $a_1 \leqs 2$.

Suppose $a_1 \geqs 3$. Set $A_i = {\rm SL}_{a_i}(q)$, $B_i = {\rm GL}_{a_i}(q)$ and write $\hat{x}=(\hat{x}_1,\hat{x}_2) \in B_1 \times B_2$. Let 
$$X = \la A_1 \times A_2, \hat{x} \ra \leqs {\rm SL}_{n}(q).$$ 
Let $x_i$ be the automorphism of the simple group $A_i/Z(A_i)$ induced by $\hat{x}_i$. The minimality of $a_1$ implies that neither $\hat{x}_1$ nor $\hat{x}_2$ is a scalar, so each $x_i$ is nontrivial. 

The key step in the proof is to study the orbits $\O_1, \ldots, \O_k$ of $X$ on $\O$, together with the action of $A_1$ and $A_2$ on each orbit. The case where $H=P_{a_1}$ can be handled directly, so assume otherwise. For convenience, let us also assume that neither $A_1/Z(A_1)$ nor $A_2/Z(A_2)$ are exceptions to the statement of the main theorem. Then using induction and a technical lemma \cite[Lemma 2.8]{LSax} one can show that $|C_{\O_j}(x)| \leqs 4|\O_j|/3q$ for each $j$, which implies that ${\rm fpr}(x) \leqs 4/3q$, a contradiction.

\vs

We now have $a_1 \leqs 2$ and $a_2 \geqs 4$ since $n \geqs 6$. By considering the orbits of $A_2$ on $\O$ and applying induction, one can reduce to the case where $A_2 \leqs H$. From here it follows that $H=P_1$ or $P_2$ (using work of Kantor \cite{Kan}, given the fact that $H$ contains long root elements of $G$), and it is not too difficult to eliminate these two possibilities.
\hspace*{\fill}$\Box$ 

\vspace{2mm}

Theorem \ref{t:liesax} plays a central role in the proof of \cite[Theorem 2]{LSax}, which yields the aforementioned lower bound $\mu(G) \geqs 2(\sqrt{n}-1)$ on the minimal degree of a primitive group $G$ of degree $n$ that does not contain ${\rm Alt}(n)$. To derive this bound, it suffices to show that $\mu(G) \geqs n/3$ unless $G$ satisfies the conditions in part (ii)(a) of Theorem \ref{t:gm}. We briefly sketch the argument.

\vs

Consider a counterexample $G \leqs {\rm Sym}(\O)$ of minimal order, with point stabiliser $H$. Fix $x \in H$ of prime order such that ${\rm fpr}(x) > 2/3$. By applying the O'Nan-Scott theorem, we can reduce to the case where $G$ is almost simple. For example, if $G$ is either an affine group or a twisted wreath product, then $N={\rm soc}(G)$ is regular (that is, $H \cap N = 1$), so 
$$|C_{\O}(x)| = |C_N(x)| \leqs \frac{|N|}{2} = \frac{n}{2}$$
and thus ${\rm fpr}(x) \leqs 1/2$, a contradiction.

Now assume $G$ is almost simple with socle $G_0$. If $G_0$ is a simple group of Lie type over $\mathbb{F}_{q}$ then Theorem \ref{t:liesax} immediately gives $\mu(G)>n/2$ if $q>2$, and $\mu(G) \geqs n/3$ if $q=2$. If $G_0$ is a sporadic group and $\mu(G)>n/2$ then Lemma \ref{l:basic}(iv) implies that 
$$1+|\chi(x)| \geqs \frac{1+\chi(1)}{2}$$ 
for some non-linear character $\chi \in {\rm Irr}(G)$. By inspecting the relevant character tables in the Atlas \cite{atlas}, one checks that no such character exists. 

Finally, suppose $G_0={\rm Alt}(m)$ is an alternating group and consider the action of $H$ on $\Gamma = \{1, \ldots, m\}$. The situation where $H$ is intransitive or imprimitive on $\Gamma$ can be handled directly, working with a concrete description of the action of $G$ on subsets or partitions. Suppose $H$ is primitive. The minimality of $|G|$ implies that $\mu(H) \geqs m/3$ with respect to the action of $H$ on $\Gamma$. This immediately translates into a lower bound of the form $|x^G| \geqs f(m)$ for some function $f$ and thus $|H|>\frac{2}{3}f(m)$ since ${\rm fpr}(x)>2/3$. But $H$ is a primitive group of degree $m$, so $|H|<g(m)$ for some function $g$ (for instance, we can take $g(m)=4^m$ by a theorem of Praeger and Saxl \cite{PS}). Together, these bounds imply that $m \leqs 750$ and by inspecting lists of small degree primitive groups one can reduce this to $m \leqs 24$. The remaining possibilities can be eliminated one-by-one. 

\subsection{Classical groups}

As observed in Remark \ref{r:ls}, the upper bound in Theorem \ref{t:liesax} is essentially best possible. However, it would be desirable to have bounds on ${\rm fpr}(x)$ that depend on the element $x$ in some way. We might also try to establish stronger bounds, at the expense of excluding some specific actions. In this section we report on recent work in this direction for  almost simple classical groups.

Let $G \leqs {\rm Sym}(\O)$ be an almost simple primitive classical group over $\mathbb{F}_{q}$ with socle $G_0$ and point stabiliser $H$. Let $V$ be the natural module for $G_0$ and set $n=\dim V$. Write $q=p^f$ with $p$ prime. The possibilities for $G_0$ are recorded in Table \ref{tab:class}. Note that we may assume the given conditions on $n$ and $q$ due to several  exceptional isomorphisms among the low-dimensional classical groups (see \cite[Proposition 2.9.1]{KL} for example). 

\renewcommand{\arraystretch}{1.1}
\begin{table}[h]
$$\begin{array}{lll} \hline\hline
\mbox{Type} & \hspace{5mm} \mbox{Notation} & \hspace{1mm} \mbox{Conditions} \\ \hline
\mbox{Linear} & \hspace{5mm} {\rm PSL}_n(q) &  \hspace{1mm} n \geqs 2, \; (n,q) \neq (2,2), (2,3) \\
\mbox{Unitary} & \hspace{5mm} {\rm PSU}_n(q) &  \hspace{1mm} n \geqs 3, \; (n,q) \neq (3,2) \\
\mbox{Symplectic} & \hspace{5mm} {\rm PSp}_n(q)' &  \hspace{1mm} \mbox{$n \geqs 4$ even} \\
\mbox{Orthogonal} & \left\{\begin{array}{l}
\O_n(q) \\
{\rm P\Omega}_n^{\pm}(q) 
\end{array}\right. & \begin{array}{l} 
\mbox{$nq$ odd, $n \geqs 7$} \\
\mbox{$n \geqs 8$ even} \end{array} \\ \hline\hline
\end{array}$$
\caption{The finite simple classical groups}
\label{tab:class}
\end{table}
\renewcommand{\arraystretch}{1}

Since $G$ is primitive, $H$ is a maximal subgroup of $G$ with $G=G_0H$. The possibilities for $H$ are described by a fundamental theorem of Aschbacher. In \cite{asch}, Aschbacher introduces eight \emph{geometric} families of subgroups of $G$, denoted by $\mathcal{C}_1, \ldots, \mathcal{C}_8$, which are defined in terms of the underlying geometry of $V$. For example, these collections include the stabilisers of certain types of subspaces of $V$, and the stabilisers of appropriate direct sum and tensor product decompositions. Roughly speaking, Aschbacher's main theorem implies that $H$ is either contained in one of the $\mathcal{C}_i$ collections, or $H$ is almost simple and the socle of $H$ acts absolutely irreducibly on $V$. Following \cite{KL}, we use $\mathcal{S}$ to denote the latter collection of \emph{non-geometric} subgroups. In turn, we write $\mathcal{S} = \mathcal{S}_{1} \cup \mathcal{S}_{2}$ where a subgroup $H \in \mathcal{S}$ is in $\mathcal{S}_{1}$ if its socle is a group of Lie type in the defining characteristic $p$. A brief description of these subgroup collections is presented in Table \ref{tab00}. 

Some further conditions are imposed on the subgroups in $\mathcal{S}$ to avoid containment in a geometric subgroup collection. For instance, suppose $G_0 = {\rm PSL}_{n}(q)$ and $H \in \mathcal{S}$ has socle $H_0$. Let 
$$\rho:\what{H}_{0} \to {\rm GL}(V)$$ 
be the corresponding absolutely irreducible representation (where $\what{H}_{0}$ is the full covering group of $H_0$). Then $\rho(\what{H}_{0})$ does not fix a non-degenerate form on $V$ and the representation cannot be realised over a proper subfield of $\mathbb{F}_{q}$ (see \cite[p.3]{KL} for a complete list of the conditions satisfied by $\rho(\what{H}_{0})$).

It turns out that a small additional subgroup collection (denoted by $\mathcal{N}$) arises when $G_0={\rm PSp}_{4}(q)'$ (with $q$ even) or ${\rm P\O}_{8}^{+}(q)$, due to the existence of certain exceptional automorphisms (the maximal subgroups in the latter case have been determined by Kleidman \cite{K08}). 

\renewcommand{\arraystretch}{1.1}
\begin{table}[h]
$$\begin{array}{cl} \hline\hline
\mathcal{C}_1 & \mbox{Stabilisers of subspaces, or pairs of subspaces, of $V$} \\
\mathcal{C}_2 & \mbox{Stabilisers of decompositions $V=\bigoplus_{i=1}^{t}V_i$, where $\dim V_i  = a$} \\
\mathcal{C}_3 & \mbox{Stabilisers of prime degree extension fields of $\mathbb{F}_{q}$} \\
\mathcal{C}_4 & \mbox{Stabilisers of decompositions $V=V_1 \otimes V_2$} \\
\mathcal{C}_5 & \mbox{Stabilisers of prime index subfields of $\mathbb{F}_{q}$} \\
\mathcal{C}_6 & \mbox{Normalisers of symplectic-type $r$-groups, $r \neq p$} \\
\mathcal{C}_7 & \mbox{Stabilisers of decompositions $V=\bigotimes_{i=1}^{t}V_i$, where $\dim V_i  = a$} \\
\mathcal{C}_8 & \mbox{Stabilisers of nondegenerate forms on $V$} \\ 
\mathcal{S} & \mbox{Almost simple absolutely irreducible subgroups} \\
\mathcal{N} & \mbox{Novelty subgroups ($G_0 = {\rm P\O}_{8}^{+}(q)$ or ${\rm PSp}_{4}(q)'$ ($p=2$), only)} \\ \hline\hline
\end{array}$$
\caption{Aschbacher's subgroup collections}
\label{tab00}
\end{table}
\renewcommand{\arraystretch}{1}

The definitive reference for information on the structure, maximality and conjugacy of the geometric subgroups is the book by Kleidman and Liebeck \cite{KL}. More recently, the maximal subgroups of the low-dimensional classical groups with $n \leqs 12$ have been completely determined by Bray, Holt and Roney-Dougal in \cite{BHR}.

\begin{ex}
If $G_0 = {\rm PSL}_{6}(q)$ then the subgroups comprising the geometric $\mathcal{C}_{i}$ collections are described below (note that the $\mathcal{C}_6$ and $\mathcal{C}_7$ collections are empty). Here we refer to the \emph{type} of a subgroup $H$, which provides an approximate description of the group-theoretic structure of $H$ (the precise structure is presented in \cite{BHR,KL}).

\vs

$\mathcal{C}_{1}$: Parabolic subgroups $P_m$ with $m \in \{1,2,3,4,5\}$, where $P_m=G_U$ for an $m$-dimensional subspace $U$ of $V$.

In addition, if $G$ contains a graph (or graph-field) automorphism $\tau$ of $G_0$ then $\mathcal{C}_1$ also includes the stabilisers of the form $G_{U,W}$, where $U,W$ are non-zero subspaces of $V$ with $6 = \dim U + \dim W$ and either $U \subset W$ or $V = U \oplus W$. (Note that if $G = \la G_0,\tau\ra$ and $m \ne 3$ then $P_m<G_0<G$ since $\dim U^{\tau} = 6-m$.) 

\vs

$\mathcal{C}_{2}$: Stabilisers of direct sum decompositions of the form $V=V_1 \oplus V_2$ or $V = U_1 \oplus U_2 \oplus U_3$, where $\dim V_i = 3$ and $\dim U_i = 2$. These subgroups are of type ${\rm GL}_{3}(q) \wr {\rm Sym}(2)$ and ${\rm GL}_{2}(q) \wr {\rm Sym}(3)$, respectively.

\vs

$\mathcal{C}_{3}$: Field extension subgroups of type ${\rm GL}_{3}(q^2)$ and ${\rm GL}_{2}(q^3)$.

\vs

$\mathcal{C}_{4}$: Tensor product subgroups of type ${\rm GL}_{3}(q) \otimes {\rm GL}_{2}(q)$.

\vs

$\mathcal{C}_{5}$: Subfield subgroups of type ${\rm GL}_{6}(q_0)$, where $q=q_0^k$ for some prime $k$.

\vs

$\mathcal{C}_{8}$: Classical subgroups of type ${\rm Sp}_{6}(q)$ and ${\rm O}_{6}^{\pm}(q)$ (with $q$ odd in the latter case), and also type ${\rm GU}_{6}(q_0)$ if $q=q_0^2$.

\vs

In addition, the possible socles of the subgroups in $\mathcal{S}$ are as follows (see \cite[Table 8.25]{BHR}):
$${\rm Alt}(6), \; {\rm Alt}(7), \; {\rm PSL}_{2}(11), \; {\rm M}_{12}, \; {\rm PSL}_{3}(4),\; {\rm PSU}_{4}(3), \; {\rm PSL}_{3}(q).$$
Note that the latter subgroup arises from the symmetric-square representation $S^2(W)$ of ${\rm SL}_{3}(q)$, where $W$ is the natural module for ${\rm SL}_{3}(q)$.
\end{ex}

\vs

When studying fixed point ratios for classical groups, it is natural to distinguish between those actions which permute subspaces of the natural module and those which do not. This leads us naturally to the following definition. 

\begin{defn}\label{d:sub}
Let $G \leqs {\rm Sym}(\O)$ be an almost simple primitive classical group over $\mathbb{F}_{q}$ with socle $G_0$, natural module $V$ and point stabiliser $H$. The action of $G$ on $\O$ is a \emph{subspace action} if one of the following holds for each maximal subgroup $M$ of $G_0$ containing $H \cap G_0$:
\begin{itemize}\addtolength{\itemsep}{0.2\baselineskip}
\item[{\rm (i)}] $M$ is the stabiliser in $G_0$ of a proper non-zero subspace $U$ of $V$, where $U$ is totally singular, non-degenerate or, if $G_0$ is orthogonal and $q$ is even, a 
non-singular $1$-space ($U$ can be any subspace if $G_0 = {\rm PSL}_{n}(q)$).
\item[{\rm (ii)}] $M= {\rm O}_{n}^{\pm}(q)$ if $G_0={\rm Sp}_{n}(q)$ and $q$ is even.
\end{itemize}
\end{defn}

\begin{ex}
If $G_0 = {\rm PSp}_{6}(q)$ then the subspace actions correspond to the following maximal subgroups $H$ of $G$:
\begin{itemize}\addtolength{\itemsep}{0.2\baselineskip}
\item[$\mathcal{C}_{1}$:] $H=P_m=G_U$ is a maximal parabolic subgroup, where $U$ is a totally singular $m$-space and $m \in \{1,2,3\}$.
\item[$\mathcal{C}_{1}$:] $H=G_W$ is of type ${\rm Sp}_{4}(q) \times {\rm Sp}_{2}(q)$, where $W$ is a non-degenerate $2$-space.
\item[$\mathcal{C}_{8}$:] $H$ is of type ${\rm O}_{6}^{+}(q)$ or ${\rm O}_{6}^{-}(q)$, and $q$ is even.
\end{itemize}
Note that a subgroup $H$ of the latter type is the stabiliser of a non-degenerate quadratic form on $V$. However, if we consider the isomorphism ${\rm Sp}_{6}(q) \cong {\rm O}_{7}(q)$ (where ${\rm O}_{7}(q)$ is the isometry group of a non-singular quadratic form on a $7$-dimensional space over $\mathbb{F}_{q}$), then $H$ corresponds to the stabiliser of an appropriate non-degenerate $6$-space. This explains why we include these subgroups in Definition \ref{d:sub}. 
\end{ex}

In general, notice that subspace actions correspond to maximal subgroups in the collection $\mathcal{C}_{1}$ (in addition to the special $\mathcal{C}_{8}$-subgroups that arise when $G$ is a symplectic group in even characteristic).

As previously remarked, it is sensible to make a distinction between subspace and non-subspace actions when studying fixed point ratios for classical groups. In general, the stabilisers for subspace actions tend to be large subgroups, such as maximal parabolic subgroups, so it is natural to expect that ${\rm fpr}(x) = |x^G \cap H| /|x^G|$ will also be large in this situation. For example, we demonstrated the sharpness of Theorem \ref{t:liesax} by considering the action of ${\rm PGL}_{n}(q)$ on $1$-spaces. Therefore, it is reasonable to expect that better bounds can be established if we exclude subspace actions. In addition, we have a very concrete description of subspace actions, which may permit direct calculation, so it also makes sense to treat them separately from this point of view.

\subsection{Subspace actions of classical groups}

Let $G \leqs {\rm Sym}(\O)$ be a primitive almost simple classical group over $\mathbb{F}_{q}$ in a subspace action with socle $G_0$ and natural module $V$. Fix an element $x \in G \cap {\rm PGL}(V)$ and write 
$x=\hat{x}Z$, where $\hat{x} \in {\rm GL}(V)$ and $Z$ denotes the centre of ${\rm GL}(V)$. Since we can identify $\O$ with a collection of subspaces (or pairs of subspaces) of $V$, it is natural to expect that ${\rm fpr}(x)$ will reflect certain properties of the action of $x$ on $V$. For instance, in Example \ref{ex:gl} we observed that ${\rm fpr}(x) = q^{d-n}$ for the natural action of ${\rm GL}_{n}(q)$ on $V=\mathbb{F}_{q}^n$, where $d = \dim C_V(x)$ is the dimension of the $1$-eigenspace of $x$ on $V$. To formalise this, we introduce the following notation (recall that if $y \in {\rm GL}(W)$, then $[W,y]$ is the subspace of $W$ spanned by the vectors of the form $w-wy$, for $w \in W$).

\begin{defn}\label{d:nu}
For $x \in {\rm PGL}(V)$, let $\hat{x}$ be a pre-image of $x$ in ${\rm GL}(V)$ and define 
$$\nu(x) = \min\{\dim[\bar{V} , \lambda \hat{x}] \,:\, \lambda \in K^{\times}\},$$
where $\bar{V} = V \otimes K$ and $K$ is the algebraic closure of $\mathbb{F}_{q}$. Note that $\nu(x)$ is equal to the codimension of the largest eigenspace of $\hat{x}$ on $\bar{V}$.
\end{defn}

\begin{ex}
Consider the action of $G={\rm PSp}_{n}(q)$ on the set $\O$ of $2$-dimensional non-degenerate subspaces of $V$. Assume $q$ is odd and set $x=\hat{x}Z$, where $\hat{x}=[-I_{m},I_{n-m}]$ and $0<m<n/2$ with respect to an appropriate symplectic basis of $V$. The eigenspaces $U$ and $W$ of $\hat{x}$ are non-degenerate, so $m=\dim U = \nu(x)$ is even and $x$ stabilises the orthogonal decomposition $V = U \perp W$. Since $G$ acts transitively on $\O$ we have 
$$|\O| = \frac{|{\rm Sp}_{n}(q)|}{|{\rm Sp}_{2}(q)||{\rm Sp}_{n-2}(q)|} \sim q^{2(n-2)}.$$
Clearly, $x$ fixes a non-degenerate $2$-space if and only if it is contained in either $U$ or $W$, so 
\begin{align*}
|C_{\O}(x)| =  \frac{|{\rm Sp}_{m}(q)|}{|{\rm Sp}_{2}(q)||{\rm Sp}_{m-2}(q)|} + \frac{|{\rm Sp}_{n-m}(q)|}{|{\rm Sp}_{2}(q)||{\rm Sp}_{n-m-2}(q)|} & \sim q^{2(m-2)} + q^{2(n-m-2)} \\
& \sim q^{2(n-m-2)}
\end{align*}
and thus
${\rm fpr}(x) \sim q^{-2m} = q^{-2\nu(x)}$.
\end{ex}

The most general result for subspace actions is the following theorem of Frohardt and Magaard \cite{FM_sub}, which shows that the previous example is typical for all subspace actions.

\begin{thm}\label{t:fm_sub}
Fix $\e>0$ and let $G \leqs {\rm Sym}(\O)$ be a primitive almost simple classical group over $\mathbb{F}_{q}$ with natural module $V$, where $\O$ is an appropriate set of $k$-subspaces of $V$. Then there exists an integer $N = N(q,\e)$ such that if $\dim V \geqs N$ then 
$${\rm fpr}(x)< q^{-\nu(x)k}+\e$$
for all $1 \ne x \in G \cap {\rm PGL}(V)$. 
\end{thm}

This is a somewhat simplified version of their main result. Indeed, \cite{FM_sub} provides a suitably modified version of the theorem that holds for all non-identity elements in $G$, together with explicit upper and lower bounds on ${\rm fpr}(x)$. For instance, \cite[Theorem 1]{FM_sub} states that if $G_0 = {\rm PSL}_{n}(q)$, $n \geqs 5$ and $\O$ is the set of $k$-dimensional subspaces of $V$, then either
\begin{itemize}\addtolength{\itemsep}{0.2\baselineskip}
\item[{\rm (a)}] ${\rm fpr}(x) \leqs 9q^{-(n-1)/2}$, or
\item[{\rm (b)}] $x \in G \cap {\rm PGL}(V)$, $\nu(x) \leqs n/2k$ and
$$q^{-\nu(x)k} - q^{-n} \leqs {\rm fpr}(x) \leqs q^{-\nu(x)k}+11q^{-n/2}.$$
\end{itemize}

We also refer the reader to \cite[Section 3]{GK} for some alternative upper bounds on $\max_{1 \ne x \in G}{\rm fpr}(x)$ for subspace actions (this work of Guralnick and Kantor was motivated by very different applications, which we will discuss in Section \ref{s:gen}).

\subsection{Non-subspace actions of classical groups}

Now let us turn to the non-subspace actions of classical groups. Here it is natural to distinguish between geometric and non-geometric actions. For geometric actions, we have a rather concrete description of the embedding of $H=G_{\a}$ in $G$, which permits a detailed analysis of the conjugacy classes in $H$ and, more importantly, their fusion in $G$. In this way, it is possible to compute accurate fixed point ratio estimates for geometric actions. 

\begin{ex}
Suppose $G = {\rm PSp}_{12}(q)$ and $H$ is a $\mathcal{C}_{2}$-subgroup of $G$ of type 
${\rm Sp}_{4}(q) \wr {\rm Sym}(3)$, so $H$ is the stabiliser of an orthogonal decomposition $V = V_1 \perp V_2 \perp V_3$ and each $V_i$ is a nondegenerate $4$-space. Let $Z$ denote the centre of ${\rm Sp}_{12}(q)$. Assume $q \equiv 1 \imod{3}$ and set $x = \hat{x}Z \in G$ where $\hat{x} = [I_{6}, \l I_3, \l^2I_3]$ and $\l \in \mathbb{F}_{q}$ is a primitive cube root of unity (here we are thinking of $\hat{x}$ as a diagonal matrix, with respect to an appropriate basis). Since two semisimple elements in a symplectic group are conjugate if and only if they have the same eigenvalues (in a splitting field), we see that $x^G \cap H = x_1^H \cup x_2^H$ where $\hat{x}_1,\hat{x}_2 \in {\rm Sp}_{4}(q)^3$ are as follows:
$$\hat{x}_1 = ([I_4], [I_2, \l, \l^2],[\l I_2, \l^2I_2]), \;\; 
\hat{x}_2 = ([I_2, \l, \l^2],[I_2, \l, \l^2], [I_2, \l, \l^2]).$$
Therefore
\begin{align*}
|x^G \cap H| & = 3! \cdot \frac{|{\rm Sp}_{4}(q)|}{|{\rm Sp}_{2}(q)||{\rm GL}_{1}(q)|} \cdot \frac{|{\rm Sp}_{4}(q)|}{|{\rm GL}_{2}(q)|} + \left(\frac{|{\rm Sp}_{4}(q)|}{|{\rm Sp}_{2}(q)||{\rm GL}_{1}(q)|}\right)^3 \\
& \sim 6q^{12}+q^{18}
\end{align*}
and
$$|x^G| = \frac{|{\rm Sp}_{12}(q)|}{|{\rm Sp}_{6}(q)||{\rm GL}_{3}(q)|} \sim q^{48},$$
so ${\rm fpr}(x) \sim q^{-30} \sim |x^G|^{-5/8}$ (we refer the reader to \cite[Chapter 3]{BG_book} for detailed information on the centralisers of elements of prime order in finite classical groups).
\end{ex}

We require different methods to handle the non-geometric actions of classical groups. Indeed, in general we are unable to determine the maximal subgroups $H \in \mathcal{S}$ for a given classical group $G$. Of course, we do not even know the dimensions of the irreducible representations of simple groups, let alone information on the embedding of $H$ in $G$ that might allow us to understand the fusion of the relevant conjugacy classes! However, as described below in Section \ref{sss:s}, there are ways to overcome these obstacles for the purposes of estimating fixed point ratios.

A key result on non-subspace actions is the following theorem of Liebeck and Shalev \cite{LSh2}, which plays a major role in several important applications. A nice feature of this result is that the upper bound depends on the size of the conjugacy class of the element.

\begin{thm}\label{t:lsh}
There is an absolute constant $\e>0$ such that 
$${\rm fpr}(x) < |x^G|^{-\e}$$
for all $x \in G$ of prime order and for every primitive almost simple classical group $G$ in a non-subspace action.
\end{thm}

It is easy to see that this result does \emph{not} extend to subspace actions. For example, if we consider the action of $G = {\rm PGL}_{n}(q)$ on $1$-spaces and we choose $x \in G$ with $\nu(x)=1$ then $|x^G| \sim q^{2n-2}$ but ${\rm fpr}(x) \sim q^{-1}$. In particular, Theorem \ref{t:lsh} implies that ${\rm fpr}(x)$ tends to $0$ as $|G|$ tends to infinity (for classical groups acting on subspaces of a fixed dimension, we only get this limiting behaviour as the field size tends to infinity).

The constant $\e$ in Theorem \ref{t:lsh} is undetermined and with applications in mind it is desirable to pin down an explicit estimate. The main theorem of \cite{Bur_1} implies that $\e \sim 1/2$ is optimal.

\begin{thm}\label{t:bur}
Let $G$ be a primitive almost simple classical group in a non-subspace action with point stabiliser $H$ and natural module of dimension $n$. Then
$${\rm fpr}(x) < |x^G|^{-\frac{1}{2}+\eta}$$
for all $x \in G$ of prime order, where $\eta \to 0$ as $n \to \infty$.
\end{thm}

This is a simplified version of \cite[Theorem 1]{Bur_1}, which is proved in the sequence of papers \cite{Bur_2,Bur_3,Bur_4}. Indeed, one can take $-1/2+1/n+\delta$ for the exponent, where $\delta=0$, or $(G,H,\delta)$ is one of a small number of known exceptions (in every case, $\delta \to 0$ as $n \to \infty$). The next example shows that there is not much room for improvement in this exponent. 

\begin{ex}
Suppose $G={\rm PSL}_{n}(q)$ and $H$ is a $\mathcal{C}_8$-subgroup of type ${\rm O}_{n}^{+}(q)$, so $n$ is even and $q$ is odd. Let $x \in G$ be an involution such that $\hat{x} = [-I_m,I_{n-m}]$ with $m$ even. Then 
$$|x^G \cap H| = \frac{|{\rm O}_{n}^{+}(q)|}{|{\rm O}_{m}^{+}(q)||{\rm O}_{n-m}^{+}(q)|} + \frac{|{\rm O}_{n}^{+}(q)|}{|{\rm O}_{m}^{-}(q)||{\rm O}_{n-m}^{-}(q)|}  \sim q^{m(n-m)}$$
and
$$|x^G| = \frac{|{\rm GL}_{n}(q)|}{|{\rm GL}_{m}(q)||{\rm GL}_{n-m}(q)|} \sim q^{2m(n-m)}$$
so ${\rm fpr}(x) \sim q^{-m(n-m)} \sim |x^G|^{-1/2}$.
\end{ex}

There are many other examples that demonstrate the accuracy of the bound in Theorem \ref{t:bur}. For instance, if $q=q_0^2$ and $H$ is a subfield subgroup of $G$ defined over $\mathbb{F}_{q_0}$ then $|x^G \cap H| \sim |x^G|^{1/2}$ for all $x \in G$ with fixed points, so ${\rm fpr}(x) \sim |x^G|^{-1/2}$. 

The proof of Theorem \ref{t:bur} is given in \cite{Bur_2, Bur_3, Bur_4}. To handle the relevant geometric actions we combine detailed information on the structure of the maximal geometric subgroups in \cite{KL} (which is organised according to the subgroup collections arising in Aschbacher's theorem) with a careful analysis of the conjugacy classes and fusion of elements of prime order.

A different approach is needed to deal with the non-geometric actions corresponding to the maximal subgroups in the collection $\mathcal{S}$. We will briefly describe the main ingredients in the next section.

\subsection{$\mathcal{S}$-actions of classical groups}\label{sss:s}

Let $G,H$ and $n$ be given as in the statement of Theorem \ref{t:bur}. Let $G_0$ be the socle of $G$, which is a simple classical group over $\mathbb{F}_{q}$ (for $q=p^f$, $p$ prime) with natural module $V$ of dimension $n$. Assume $H \in \mathcal{S}$ has socle $H_0$ and let 
$\rho:\what{H}_{0} \to {\rm GL}(V)$
be the corresponding absolutely irreducible representation.

If $n$ is small, say $n \leqs 5$, then the possibilities for $(G,H,\rho)$ are well known (see \cite{BHR}) and it is straightforward to work directly with the representation $\rho$ (and its Brauer character) to obtain sufficient information on the fusion of $H$-classes in $G$ to compute (or accurately estimate) fixed point ratios.

Now assume $n \geqs 6$. Let us write $H \in \mathcal{A}$ if $q=p$, $H_0 = {\rm Alt}(m)$ is an alternating group and $V$ is the fully deleted permutation module for $H_0$ over $\mathbb{F}_{q}$ (in which case $n=m-2$ or $m-1$). We can now state the following result, which combines the main theorem of \cite{Lieb} with \cite[Theorem 7.1]{GurSax}.

\begin{thm}\label{t:gl}
Let $G$ be a primitive almost simple classical group over $\mathbb{F}_{q}$ with socle $G_0$, point stabiliser $H \in \mathcal{S} \setminus \mathcal{A}$ and natural module $V$ of dimension $n \geqs 6$. Let $\rho:\what{H}_{0} \to {\rm GL}(V)$ be the corresponding representation. Then the following hold:
\begin{itemize}\addtolength{\itemsep}{0.2\baselineskip}
\item[{\rm (i)}] $|H|<q^{3n\a}$, where $\a=2$ if $G_0$ is unitary, otherwise $\a=1$;
\item[{\rm (ii)}] Either $\nu(x) > \max\{2,\sqrt{n}/2\}$ for all $1 \ne x \in H \cap {\rm PGL}(V)$, or $n \leqs 10$ and $(G,H,\rho)$ belongs to a short list of known exceptions. 
\end{itemize}
\end{thm}

\begin{rem}\label{r:2n}
\mbox{ } 
\begin{itemize}\addtolength{\itemsep}{0.2\baselineskip}
\item[(a)] The bound in part (i) of the theorem can be sharpened, at the expense of some additional (known) exceptions. For instance, see 
\cite[Theorem 4.2]{Lieb} and \cite[Theorem 2.10]{BGS} for improvements with $q^{3n\a}$ replaced by $q^{(2n+4)\a}$ and $q^{2n+4}$, respectively (for example, the case $(G,H) = ({\rm PSL}_{27}(q), E_6(q))$ is an exception to the bound $|H|<q^{2n+4}$).

\item[(b)] We can view the bound in (ii) as a linear analogue of the aforementioned bounds of Babai, Liebeck and Saxl on the minimal degree of a primitive permutation group (with irreducibility in place of primitivity); see Section \ref{ss:md}.

\item[(c)] The bound in (ii) is close to best possible if we impose the condition that the only exceptions occur in small dimensions. To see this, suppose $n=m^2$ where $m \geqs 3$ is odd. If $q$ is chosen appropriately then $G = {\rm PSL}_{n}(q)$ has a maximal subgroup $H \in \mathcal{S}$ with socle $H_0={\rm PSL}_{m}(q^2)$, which is embedded in $G$ via the module $W \otimes W^{(q)}$ for $\what{H}_{0}={\rm SL}_{m}(q^2)$, where $W$ is the natural module for $H_0$ and $W^{(q)}$ is the $q$-power Frobenius twist of $W$. If we take $x=[-I_{m-1}, I_1] \in H_0$ then it is easy to check that $\nu(x)=2m-2 < 2\sqrt{n}$.
\end{itemize}
\end{rem}

The proof of Theorem \ref{t:lsh} also uses the bound in part (i) of Theorem \ref{t:gl}, but the bound in (ii) is a crucial new ingredient in the proof of Theorem \ref{t:bur}.  The cases in $\mathcal{A}$, and also the small number of low-dimensional exceptions arising in part (ii) of Theorem \ref{t:gl}, are well understood embeddings and they can be handled directly. 

Generically, Theorem \ref{t:gl} tells us that $H$ is small \emph{and} the elements in $H \cap {\rm PGL}(V)$ have relatively small eigenspaces on $V$. In particular, the latter property yields a lower bound $|x^G| \geqs f(n,q)$ for all $x \in H \cap {\rm PGL}(V)$ of prime order, so we get  
\begin{equation}\label{e:hx}
{\rm fpr}(x)= \frac{|x^G\cap H|}{|x^G|} < \frac{|H|}{|x^G|} <\frac{q^{3n\a}}{f(n,q)}.
\end{equation}
Note that if $x \in H \setminus {\rm PGL}(V)$ has prime order then $x$ is either a field, graph or graph-field automorphism of $G_0$ and it is straightforward to check that the inequality $|x^G| \geqs f(n,q)$ still holds, so \eqref{e:hx} is valid for all $x \in H$ of prime order. 

\begin{ex}
Suppose $G = {\rm PSL}_{n}(q)$, $H \in \mathcal{S} \setminus \mathcal{A}$ and $n > 10$. Let $x \in H$ be an element of prime order with $\nu(x)=s$, so $s \geqs \lceil \sqrt{n}/2 \rceil=\b$ by Theorem \ref{t:gl}(ii). It is not too difficult to show that  
$$|x^G|>\frac{1}{2}q^{2\b(n-\b)}$$
(see \cite[Corollary 3.38]{Bur_2}) so we get ${\rm fpr}(x)<|x^G|^{-1/2}$ if
$$q^{6n} < \frac{1}{2}q^{2\b(n-\b)}.$$
One checks that this inequality holds if $n>36$, so we may assume that $n \leqs 36$. In fact, if we replace the bound in part (i) of Theorem \ref{t:gl} by $|H|<q^{2n+4}$ (at the expense of a small number of known exceptions (see \cite[Theorem 4.2]{Lieb}), which can be handled separately), then we can reduce to the case where $n \leqs 16$. At this point we can turn to results of L\"{u}beck \cite{Lu} (in defining characteristic) and Hiss and Malle \cite{HM, HM2} (in non-defining characteristic) to determine the possibilities for $(G,H,\rho)$ and we can then work directly with these cases.
\end{ex}

\subsection{Exceptional groups}\label{ss:ex}

Finally, let us say a few words on fixed point ratios for exceptional groups. Let $G \leqs {\rm Sym}(\O)$ be a primitive almost simple exceptional group of Lie type over $\mathbb{F}_{q}$ with socle $G_0$ and point stabiliser $H$. Recall that Theorem \ref{t:liesax} gives
$$\max_{1 \ne x \in G}{\rm fpr}(x) \leqs \frac{4}{3q}$$
and it is natural to ask if this upper bound can be improved. 

In \cite{FM_ex}, Frohardt and Magaard obtain close to best possible upper bounds in the special case where the rank of $G$ is at most $2$. For example, they prove that 
$$\max_{1 \ne x \in G}{\rm fpr}(x) = \left\{\begin{array}{ll}
\frac{1}{q^2-q+1} & \mbox{if $G_0 \in \{ G_2(q), {}^2G_2(q)\}$ and $q>4$,} \\
\frac{1}{q^4-q^2+1} & \mbox{if $G_0 = {}^3D_4(q)$.}
\end{array}\right.$$

In \cite{LLS2}, using different methods, Lawther, Liebeck and Seitz present a detailed analysis of fixed point ratios for all the exceptional groups. For instance, \cite[Theorem 1]{LLS2} gives 
$$\max_{1 \ne x \in G}{\rm fpr}(x) \leqs \left\{\begin{array}{ll}
\frac{1}{q^8(q^4-1)} & \mbox{if $G_0 = E_8(q)$} \\
\frac{1}{q^6-q^3+1} & \mbox{if $G_0 \in\{ E_7(q), {}^2E_6(q)\}$} \\
\frac{1}{q^4-q^2+1} & \mbox{if $G_0 \in \{E_6(q), F_4(q)\}$,}
\end{array}\right.$$
with equality if $G_0=E_6(q)$, ${}^2E_6(q)$ or $F_4(q)$. More detailed bounds are given in \cite[Theorem 2]{LLS2}, which depend not only on $G$, but also on the choice of $H$ and $x$ to some extent. For example, if $G=E_8(q)$ and $H$ does not contain a maximal torus of $G$, then \cite[Theorem 2]{LLS2} states that ${\rm fpr}(x) \leqs q^{-48}$ for all non-identity semisimple elements $x \in G$.

As for classical groups, the proofs rely on detailed information on the subgroup structure and conjugacy classes of the finite exceptional groups. In particular, there is a fundamental reduction theorem for subgroups due to Liebeck and Seitz, which plays a similar role to Aschbacher's theorem for classical groups (see \cite[Theorem 8]{LS03}, for example). We finish by highlighting two other important ingredients in \cite{LLS2}.

\subsubsection{Parabolic actions} The special case where $H$ is a maximal parabolic subgroup is studied using tools from the character theory of finite groups of Lie type, such as the Deligne-Lusztig theory and Green functions. These sophisticated techniques can be used to obtain very precise fixed point ratio estimates. 

For example, suppose $G=E_8(q)$, $H=P_8$ and $x \in G$ is unipotent (here our notation indicates that $H$ corresponds to the $8$-th node in the Dynkin diagram of $G$, labelled in the usual way, so the Levi factor of $H$ is of type $E_7(q)$). Then $|\O| \sim q^{57}$ and one can show that the corresponding permutation character admits the decomposition
\begin{align*}
1^G_H(x) & = \sum_{\phi \in \what{W}}n_{\phi}R_{\phi}(x) \\
& = R_{\phi_{1,0}}(x)+R_{\phi_{8,1}}(x)+R_{\phi_{35,2}}(x)+R_{\phi_{112,3}}(x) + R_{\phi_{84,4}}(x),
\end{align*}
where $\what{W} = {\rm Irr}(W)$ and $W$ is the Weyl group of $G$. The $R_{\phi}$ are almost characters of $G$ and the coefficients are given by $n_{\phi} = \la 1_{W_P}^{W}, \phi\ra$, where $W_P$ is the corresponding parabolic subgroup of $W$. The restriction of the $R_{\phi}$ to unipotent elements $x \in G$ are called \emph{Green functions}; each $R_{\phi}(x)$ is a polynomial in $q$ with non-negative coefficients. L\"{u}beck has implemented an algorithm of Lusztig to compute the relevant Green functions (modulo a sign issue for certain  elements) and his calculations yield very precise estimates for ${\rm fpr}(x)$ (see \cite[Section 2]{LLS2} for more details). For example, if $x \in G$ is a long root element then $1^G_H(x)$ can be computed precisely in this way; we get a certain monic polynomial in $q$ of degree $45$, which implies that ${\rm fpr}(x) \leqs 1/q^8(q^4-1)$. This turns out to be the largest fixed point ratio for any non-identity element of $G$.

\vs

\subsubsection{Algebraic groups} Results on the dimensions of fixed point spaces for primitive actions of exceptional algebraic groups also play a key role in \cite{LLS2}. In the general set up, $\bar{G}$ is a simple algebraic group over the algebraic closure $K=\bar{\mathbb{F}}_{q}$ and $\s$ is a Frobenius morphism of $\bar{G}$ such that $G_0$ is the derived subgroup of $\bar{G}_{\s} = \{x \in \bar{G} \,:\, x^{\s}=x\}$. Let $\bar{H}$ be a $\s$-stable closed subgroup of $\bar{G}$ and let 
$\Gamma = \bar{G}/\bar{H}$ be the corresponding coset variety, which is naturally a $\bar{G}$-variety over $K$. Then the fixed point space $C_{\Gamma}(x)$ is a subvariety for each $x \in \bar{G}$, and we may compare the dimensions of 
$\Gamma$ and $C_{\Gamma}(x)$. In analogy with Lemma \ref{l:basic}(iii), we have
$$\dim C_{\Gamma}(x) - \dim \Gamma = \dim (x^{\bar{G}} \cap \bar{H}) - \dim x^{\bar{G}}$$
for all $x \in \bar{H}$ (see \cite[Proposition 1.14]{LLS}). Moreover, if we set $H = \bar{H}_{\s}$ then
$${\rm fpr}(x) = \frac{|x^G\cap H|}{|x^G|} \sim q^{\dim(x^{\bar{G}} \cap \bar{H}) - \dim x^{\bar{G}}}$$
for all $x \in H$ (see \cite[Lemma 4.5]{LLS2}, for example). In this way, if $H$ corresponds to a $\s$-stable closed subgroup of $\bar{G}$, then it is possible to use dimension bounds at the algebraic group level to study fixed point ratios for the finite group $G$.

This interplay between finite and algebraic groups is applied repeatedly in \cite{LLS2}, using results obtained for primitive actions of exceptional algebraic groups in the companion paper \cite{LLS}. Similar considerations also play a role in the analysis of geometric actions of finite classical groups in the proof of Theorem \ref{t:bur}, using results for classical algebraic groups in \cite{Bur2}.

\section{Generation and random generation}\label{s:gen}

In this section we discuss applications of fixed point ratios to problems concerning the generation and random generation of finite groups. In particular, we will explain how fixed point ratios play a key role in the solution to Problem A on generating graphs of simple groups stated in the introduction.

\subsection{Simple groups}

Recall that a group is \emph{$n$-generated} if it can be generated by $n$ elements. For instance, dihedral and symmetric groups are $2$-generated, e.g. we have ${\rm Sym}(n) = \la (1,2), (1,2, \ldots, n) \ra$. The following theorem (essentially due to Steinberg \cite{St}) is the starting point for the investigation of many interesting problems.

\begin{thm}\label{t:sd2}
Every finite simple group is $2$-generated.
\end{thm}

The proof relies on CFSG. The alternating groups are easy: 
$${\rm Alt}(n) = \left\{\begin{array}{ll}\langle (1,2,3), (1,2, \ldots, n)\rangle & \mbox{$n$ odd,} \\
 \langle (1,2,3), (2,3, \ldots, n)\rangle & \mbox{$n$ even.}
 \end{array}\right.$$
In \cite{St}, Steinberg presents explicit generating pairs for each simple group of Lie type. For instance, ${\rm PSL}_{2}(q) = \la xZ,yZ \ra$, where $Z=Z({\rm SL}_{2}(q))$ and
$$x=\left(\begin{array}{ll}
\mu & 0 \\
0 & \mu^{-1}
\end{array}\right),\;\; y=\left(\begin{array}{ll}
-1 & 1 \\
-1 & 0
\end{array}\right)$$
with $\mathbb{F}_{q}^{\times}=\la \mu\ra$. In \cite{AG}, Aschbacher and Guralnick complete the proof of the theorem by showing that every sporadic group is $2$-generated. 

\begin{rem}
By a theorem of Dalla Volta and Lucchini \cite{DL}, every almost simple group is $3$-generated (there are such groups that really need $3$ generators, e.g. take $G = {\rm Aut}({\rm PSL}_{n}(q))$ with $nq$ odd and $q=p^{2f}$ with $p$ prime).
\end{rem}

\vs

In view of Theorem \ref{t:sd2}, it is natural to consider the abundance of generating pairs in a finite simple group (or a sequence of such groups), or the existence of generating pairs with special properties (such as prescribed orders). Problems of this flavour  have been intensively investigated in recent years. 

\subsection{Random generation} 

Let $G$ be a finite group, let $k$ be a positive integer and let
$$\mathbb{P}(G,k) = \frac{|\{(x_1, \ldots, x_k) \in G^k \, :\, G = \la x_1, \ldots, x_k \ra\}|}{|G|^k}$$
be the probability that $k$ randomly chosen elements generate $G$.  

\begin{con}[Netto \cite{Netto}, 1882] 
``If we arbitrarily select two or more substitutions of $n$ elements, it is to be regarded as extremely probable that the group of lowest order which contains these is the symmetric group, or at least the alternating group."
\end{con} 

In our terminology, Netto is claiming that $\lim_{n \to \infty}\mathbb{P}({\rm Alt}(n),2)=1$. This remarkable conjecture was proved by Dixon \cite{Dixon} in a highly influential paper published in 1969, which relies in part on the pioneering work of Erd\"{o}s and Tur\'{a}n in the mid-1960s on statistical properties of symmetric groups. 
In the same paper, Dixon makes the bold conjecture that \emph{all} finite simple groups are strongly $2$-generated in the sense of Netto.

\begin{con}\label{c:gnp}
Let $(G_n)$ be any sequence of finite simple groups such that $|G_n|$ tends to infinity with $n$. Then $\lim_{n\to \infty}\mathbb{P}(G_n,2) =1$.
\end{con}

Dixon's conjecture was eventually proved in the 1990s. In \cite{KanL}, Kantor and Lubotzky establish the conjecture for classical groups and low rank exceptional groups, and the remaining exceptional groups were handled by Liebeck and Shalev \cite{LieSh}. 

In both papers, the strategy of the proof is based on an elementary observation. Let $\mathcal{M}$ be the set of maximal subgroups of $G$ and let $x,y \in G$ be randomly chosen elements. If $G \neq \langle x,y \rangle$ then $x,y \in H$ for some $H \in \mathcal{M}$. The probability of this event is $|G:H|^{-2}$, so 
$$1-\mathbb{P}(G,2) \leqs \sum_{H \in \mathcal{M}}{|G:H|^{-2}}=: Q(G).$$
By carefully studying $\mathcal{M}$, using recent advances in our understanding of the subgroup structure of the simple groups of Lie type (such as Aschbacher's theorem for classical groups), one shows that $Q(G) \to 0$ as $|G|$ tends to infinity, and the result follows. 

Note that this probabilistic approach shows that every sufficiently large finite simple group is $2$-generated, without the need to explicitly construct a pair of generators. Many interesting related results have been established in more recent years. For example, the following result is \cite[Theorem 1.1]{MQR}.

\begin{thm}
We have $\mathbb{P}(G,2) \geqs 53/90$ for every finite simple group $G$, with equality if and only if $G={\rm Alt}(6)$.
\end{thm}

\subsection{Spread}

The following $2$-generation property was introduced by Brenner and Wiegold \cite{BW} in the 1970s.

\begin{defn}
Let $G$ be a finite group and let $k$ be a positive integer. Then $G$ has \emph{spread at least $k$} if for any non-identity elements $x_1, \ldots, x_k \in G$ there exists $y \in G$ such that $G = \la x_i,y\ra$ for all $i$. We say that $G$ is \emph{$\frac{3}{2}$-generated} if it has spread at least $1$.
\end{defn}

We will also be interested in the more restrictive notion of \emph{uniform spread}, which was introduced more recently in \cite{BGK}.

\begin{defn}
We say that $G$ has \emph{uniform spread at least $k$} if there exists a fixed conjugacy class $C$ of $G$ such that for any non-identity elements $x_1, \ldots, x_k \in G$ there exists $y \in C$ such that $G = \la x_i,y\ra$ for all $i$.
\end{defn}

Clearly, every cyclic group has uniform spread at least $k$ for all $k \in \mathbb{N}$, so for the remainder of this discussion let us assume $G$ is non-cyclic. Set
\begin{align*}
s(G) & = \max\{k \in \mathbb{N}_{0} \, : \, \mbox{$G$ has spread at least $k$}\} \\
u(G) & = \max\{k \in \mathbb{N}_{0} \, : \, \mbox{$G$ has uniform spread at least $k$}\}
\end{align*}
(so $u(G)=0$ if $G$ does not have uniform spread at least $1$, etc.). Note that $u(G) \leqs s(G) < |G|-1$ and there are examples with $u(G)<s(G)$. For example, if $G = {\rm Sym}(6)$ then $u(G)=0$ and $s(G)=2$. Note that $G$ is $\frac{3}{2}$-generated if and only if $s(G) \geqs 1$.

In \cite{BW}, Brenner and Wiegold study the spread of the simple groups
${\rm Alt}(n)$ and ${\rm PSL}_{2}(q)$. Among other things, they prove that 
$s({\rm Alt}(2m))=4$ if $m \geqs 4$ and $s({\rm PSL}_{2}(q)) = q-2$ if $q$ is even. 

The following theorem of Breuer, Guralnick and Kantor is the main result on the spread of simple groups (see \cite[Corollary 1.3]{BGK}).

\begin{thm}\label{t:bgk}
Let $G$ be a nonabelian finite simple group. Then $u(G) \geqs 2$, with equality if and only if
\begin{equation}\label{e:gp}
G \in \{{\rm Alt}(5), {\rm Alt}(6), \O_{8}^{+}(2), {\rm Sp}_{2m}(2) \, (m \geqs 3)\}.
\end{equation}
\end{thm}

It turns out that $u(G) = s(G) = 2$ for each of the groups in \eqref{e:gp}. 

\begin{rem}
The weaker bound $u(G) \geqs 1$ was originally obtained by Stein \cite{Stein}, and independently by Guralnick and Kantor \cite{GK}. In the  latter paper, the authors prove that there is a conjugacy class $C$ of $G$ such that each non-identity element of $G$ generates $G$ with at least $1/10$ of the elements in $C$, and they also establish some related results for almost simple groups. In \cite{BGK}, the constant $1/10$ is replaced by  $13/42$ (for $G = \O_8^{+}(2)$, this is best possible). In fact, with the exception of a known finite list of small groups, plus the family of symplectic groups over $\mathbb{F}_{2}$, $1/10$ can be replaced by $2/3$ (see \cite[Theorem 1.1]{BGK}). As explained below, this result is the key ingredient in the proof of Theorem \ref{t:bgk}. We also note that in an earlier paper, Guralnick and Shalev  proved that $u(G) \geqs 2$ for all sufficiently large simple groups $G$ (see \cite[Theorem 1.2]{GSh}).
\end{rem}

Fixed point ratios play a central role in the proof of Theorem \ref{t:bgk}. Let us explain the connection.  
Let $G$ be a finite group. For $x,y \in G$, let 
$$\mathbb{P}(x,y)= \frac{|\{z \in y^G \, : \, G=\langle x,z \rangle \}|}{|y^G|}$$
be the probability that $x$ and a randomly chosen conjugate of $y$ generate $G$. Set 
$$Q(x,y) = 1- \mathbb{P}(x,y).$$

\begin{lem}\label{lp1}
Suppose there exists an element $y \in G$ and a positive integer $k$ such that 
$Q(x,y)<1/k$ for all $1 \ne x \in G$. Then $u(G) \geqs k$.
\end{lem}

\begin{proof}
Let $x_1, \ldots, x_k \in G$ be non-identity elements and let $E$ denote the event $E_1 \cap \cdots \cap E_k$, where $E_i$ is the event that $G=\langle x_i,z \rangle$ for a randomly chosen conjugate $z \in y^G$. Let $\mathbb{P}(E)$ be the probability that $E$ occurs and let $\bar{E}$ be the complementary event (and similarly for $\mathbb{P}(E_i)$ and $\bar{E}_i$). We need to show that $\mathbb{P}(E)>0$. To see this, we note that
\[\mathbb{P}(E) = 1-\mathbb{P}(\bar{E}) =  1- \mathbb{P}(\bar{E}_1 \cup \cdots \cup \bar{E}_k) \geqs 1- \sum_{i=1}^k \mathbb{P}(\bar{E}_i)  = 1- \sum_{i=1}^k Q(x_i,y) \]
so $\mathbb{P}(E) >1 - k \cdot \frac{1}{k} = 0$ 
and the result follows.
\end{proof}

Let $\mathcal{M}(y)$ be the set of maximal subgroups of $G$ containing $y$. The following result is the main tool in the proof of Theorem \ref{t:bgk}.

\begin{cor}\label{c:mgen}
Suppose there is an element $y \in G$ and a positive integer $k$ such that 
$$\sum_{H \in \mathcal{M}(y)}{\rm fpr}(x,G/H)<\frac{1}{k}$$
for all elements $x \in G$ of prime order. Then $u(G) \geqs k$.
\end{cor}

\begin{proof}
In view of Lemma \ref{lp1}, it suffices to show that
$$Q(x,y) \leqs \sum_{H \in \mathcal{M}(y)}{{\rm fpr}(x,G/H)}$$
for all $1 \ne x \in G$. Fix a non-identity element $x \in G$ and let $z \in y^G$. Then $G \neq \la x,z \ra$ if and only if $\la x',y \ra \leqs H$ for some $x' \in x^G$ and $H \in \mathcal{M}(y)$, so we have
$$Q(x,y) \leqs \sum_{H \in \mathcal{M}(y)}\mathbb{P}_{x}(H),$$
where $\mathbb{P}_{x}(H)$ is the probability that a randomly chosen conjugate of $x$ lies in $H$. Now
$$\mathbb{P}_{x}(H)= \frac{|x^G \cap H|}{|x^G|} = {\rm fpr}(x,G/H)$$
and the result follows.
\end{proof}

The key step in applying Corollary \ref{c:mgen} is to carefully choose $y \in G$ so that it belongs to very few maximal subgroups of $G$, with the essential extra property that we can explicitly determine the subgroups in $\mathcal{M}(y)$, or at least a collection of maximal subgroups containing $\mathcal{M}(y)$ that is not too much bigger.
This means that there is some flexibility in the approach -- the optimal choice of $y$ is not always obvious (in practice, it seems that there are many valid possibilities, but some will require more work than others).

\begin{ex}
Let's use this approach to prove that $u({\rm Alt}(5)) \geqs 2$. Set $G = {\rm Alt}(5)$ and $y=(1,2,3,4,5) \in G$. The maximal subgroups of $G$ are isomorphic to ${\rm Sym}(3)$, ${\rm Alt}(4)$ and $D_{10}$, and it is easy to see that $\mathcal{M}(y) = \{K\}$ with
$$K = \la (1,2,3,4,5), (1,2)(3,5) \ra =D_{10}.$$ 
We now compute
$$\sum_{H \in \mathcal{M}(y)} {\rm fpr}(x,G/H) = {\rm fpr}(x,G/K) = \left\{\begin{array}{ll}
1/3 & |x|=2\\
0 & |x|=3 \\
1/6 & |x|=5
\end{array}\right.$$
and thus $u(G) \geqs 2$ by Corollary \ref{c:mgen}. In fact, we have $u(G)=2$ (see Theorem \ref{t:bgk}), which shows that the strictness of the inequality in Corollary \ref{c:mgen} is essential. As an aside, one can check that the class of $3$-cycles has the uniform spread $1$ property, but not spread $2$.
\end{ex}

\begin{ex}
We claim that $u(G) \geqs 3$ if $G = {\rm Alt}(n)$ and $n \geqs 8$ is even (recall the result of Brenner and Wiegold, which states that $s(G)=4$). To see this, set $n=2m$ and $k=m-(2,m-1)$. Fix $y \in G$ with cycle-shape $[k,n-k]$ and note that $(k,n-k)=1$. We claim that $\mathcal{M}(y)$ consists of a single intransitive subgroup $H$ of type ${\rm Sym}(k) \times {\rm Sym}(n-k)$. It is clear that $H$ is the only intransitive subgroup in $\mathcal{M}(y)$, so assume $M \in \mathcal{M}(y)$ is transitive. The cycle-shape of $y$ implies that $M$ is primitive, but $\la y \ra$ contains a $k$-cycle and thus $M=G$ by a classical result of Marggraf (1892), which is a contradiction (Marggraf's theorem implies that the only primitive groups of degree $n$ containing a cycle of length $\ell<n/2$ are ${\rm Sym}(n)$ and ${\rm Alt}(n)$; see \cite[Theorem 13.5]{Wie}). This justifies the claim. It remains to estimate fixed point ratios  with respect to the action of $G$ on $k$-sets. A straightforward combinatorial argument shows that ${\rm fpr}(x) < 1/3$ for all $x \in G$ of prime order (see the proof of \cite[Proposition 6.3]{BGK}) and the result follows.
\end{ex}

\begin{rem}
The analysis of odd degree alternating groups is slightly more complicated. In this situation, no elements have precisely two cycles, so one is forced to work with $n$-cycles, which may belong to several maximal subgroups. 
\end{rem}

\begin{rem}
By a theorem of Guralnick and Shalev \cite[Theorem 1.1]{GSh}, if $G_i = {\rm Alt}(n_i)$ and $n_i$ tends to infinity with $i$, then $s(G_i)$ tends to infinity if and only if $p(n_i)$ tends to infinity, where $p(n_i)$ is the smallest prime divisor of $n_i$.
\end{rem}

Let us also comment on the proof of Theorem \ref{t:bgk} for classical groups, which require the most work. To illustrate some of the main ideas, we will assume that $G={\rm PSL}_{n}(q)$ and $n \geqs 13$ is odd. Following \cite{BGK}, fix a semisimple element $y \in G$ preserving a decomposition $V = U \oplus W$ of the natural module, where $\dim U = k=(n-1)/2$ and $y$ acts irreducibly on $U$ and $W$. We claim that $\mathcal{M}(y) = \{G_U,G_W\}$, which quickly implies that 
$$\sum_{H \in \mathcal{M}(y)} {\rm fpr}(x,G/H) = 2\cdot {\rm fpr}(x,G/G_U) < \frac{1}{3}$$
for all $x \in G$ of prime order (note that the actions of $G$ on $k$-subspaces and $(n-k)$-subspaces of $V$ are permutation isomorphic and ${\rm fpr}(x,G/G_U) = {\rm fpr}(x,G/G_W)$ for all $x \in G$). For example, if $q \geqs 8$ then 
${\rm fpr}(x,G/G_U) \leqs 1/6$ by Theorem \ref{t:liesax}. In particular, we conclude that $u(G) \geqs 3$.

In order to determine the subgroups in $\mathcal{M}(y)$, it is very helpful to observe that $|y|$ is divisible by a \emph{primitive prime divisor} of $q^{n-k}-1$ (that is, a prime $r$ such that $n-k$ is the smallest positive integer $i$ such that $r$ divides $q^{i}-1$; a classical theorem of Zsigmondy (1892) establishes the existence of such primes if $n-k \geqs 3$ and $(n-k,q) \ne (6,2)$). The subgroups of classical groups containing such ppd elements are studied in \cite{GPPS}, where the analysis is organised according to Aschbacher's subgroup structure theorem. The main theorem of \cite{GPPS}  severely limits the possible subgroups in $\mathcal{M}(y)$, and many of these possibilities can be ruled out by considering the order of $y$ (which is roughly $(q^n-1)/(q-1)$ since $(k,n-k)=1$). We refer the reader to the proof of \cite[Proposition 5.23]{BGK} for the details.

\subsection{Generating graphs}

The following notion first appeared in a paper by Liebeck and Shalev \cite{LSh3} on random generation.

\begin{defn}
Let $G$ be a finite group. The \emph{generating graph} $\Gamma(G)$ is a graph on the non-identity elements of $G$ so that two vertices $x,y$ are joined by an edge if and only if $G=\la x,y \ra$.
\end{defn}

\begin{figure}
\begin{center}
\begin{tikzpicture}[scale=0.6] 
 \tikzstyle{every node}=[font=\small]
\draw[blue, opacity=0.75] ( 180 :3) -- ( 45 :3);
\draw[blue, opacity=0.75] ( 225 :3) -- ( 45 :3);
\draw[blue, opacity=0.75] ( 270 :3) -- ( 45 :3);
\draw[blue, opacity=0.75] ( 315 :3) -- ( 45 :3);
\draw[blue, opacity=0.75] ( 180 :3) -- ( 135 :3);
\draw[blue, opacity=0.75] ( 225 :3) -- ( 135 :3);
\draw[blue, opacity=0.75] ( 270 :3) -- ( 135 :3);
\draw[blue, opacity=0.75] ( 315 :3) -- ( 135 :3);
\draw[blue, opacity=0.75] ( 225 :3) -- ( 180 :3);
\draw[blue, opacity=0.75] ( 315 :3) -- ( 180 :3);
\draw[blue, opacity=0.75] ( 270 :3) -- ( 225 :3);
\draw[blue, opacity=0.75] ( 315 :3) -- ( 270 :3);
\fill[black, opacity=1] (45:3) circle (2.5pt);
 \fill[black, opacity=1] (90:3) circle (2.5pt);
 \fill[black, opacity=1] (135:3) circle (2.5pt);
 \fill[black, opacity=1] (180:3) circle (2.5pt);
 \fill[black, opacity=1] (225:3) circle (2.5pt);
 \fill[black, opacity=1] (270:3) circle (2.5pt);
 \fill[black, opacity=1] (315:3) circle (2.5pt);
 \node at (45:3.4){$a$};
  \node at (90:3.4){$a^2$}; 
   \node at (135:3.4){$a^3$}; 
 \node at (180:3.4){$b$};
 \node at (225:3.4){$ab$};
 \node at (270:3.4){$a^2b$};
 \node at (315:3.4){$a^3b$};
 \end{tikzpicture}
 \end{center}
 \caption{The generating graph of $D_8 = \la a,b \mid a^4=b^2=1,\, a^b=a^{-1} \ra$}
 \label{fig:d8}
 \end{figure}
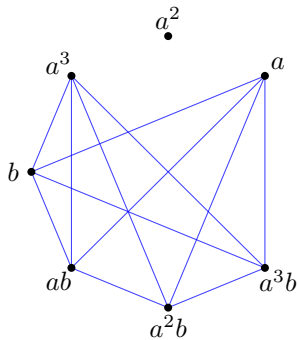
 
This graph encodes many interesting generation properties of the group. For example,
\begin{align*}
\mbox{$G$ is $2$-generated} & \iff \mbox{The edge-set of $\Gamma(G)$ is non-empty} \\
\mbox{$G$ has spread $1$} & \iff \mbox{$\Gamma(G)$ has no isolated vertices} \\
\mbox{$G$ has spread $2$} & \,\implies \mbox{$\Gamma(G)$ is connected with diameter at most $2$} 
\end{align*}
Moreover, the interplay between groups and graphs suggests many natural problems. For instance, what is the (co)-clique number and chromatic number of $\Gamma(G)$? Does $\Gamma(G)$ contain a Hamiltonian cycle?

In view of Theorem \ref{t:sd2}, it makes sense to study the generating graphs of finite simple groups. Moreover, the fact that simple groups are strongly $2$-generated (in the probabilistic sense of Netto and Dixon (see Conjecture \ref{c:gnp}), for example, or in the sense of spread, as in Theorem \ref{t:bgk}) suggests that the corresponding generating graphs should have lots of edges and therefore strong connectivity properties.
 
\begin{ex}
If $G={\rm Alt}(5)$ then $\Gamma(G)$ has $59$ vertices and one checks that there are $1140$ edges. It also has clique number $8$, chromatic number $9$ and coclique number $15$ (for example, a maximal coclique is given by the set of $15$ elements of order $2$). 
\end{ex} 
 
The following result summarises some of the main results on generating graphs for simple groups.
 
\begin{thm}
Let $G$ be a nonabelian finite simple group and let $\Gamma(G)$ be its generating graph.
\begin{itemize}\addtolength{\itemsep}{0.2\baselineskip}
\item[{\rm (i)}] $\Gamma(G)$ has no isolated vertices.
\item[{\rm (ii)}] $\Gamma(G)$ is connected and has diameter $2$.
\item[{\rm (iii)}] $\Gamma(G)$ contains a Hamiltonian cycle if $|G|$ is sufficiently large.
\end{itemize}
\end{thm}

\begin{proof}
Clearly, (ii) implies (i), and (ii) is an immediate corollary of Theorem \ref{t:bgk}. Part (iii) is one of the main results in \cite{BGLMN} and we briefly sketch the argument in the case where $G$ is a group of Lie type. The three main ingredients are as follows:
\begin{itemize}\addtolength{\itemsep}{0.2\baselineskip}
\item[{\rm (a)}] By the proof of Dixon's conjecture, we know that $\mathbb{P}(G,2) \to 1$ as $|G|$ tends to infinity. More precisely, a theorem of Liebeck and Shalev (see \cite[Theorem 1.6]{LSh3}) states that there is a positive constant $c_1$ such that 
$$\mathbb{P}(G,2) \geqs 1 - \frac{c_1}{m(G)}$$
for any nonabelian finite simple group $G$, where $m(G)$ is the minimal index of a proper subgroup of $G$. Note that $G \leqs {\rm Sym}(m(G))$, so $m(G)$ tends to infinity with $|G|$.

\item[{\rm (b)}] Set $|G|=m+1$ (so $m$ is odd) and let $d_i$ be the degree of the $i$-th vertex of $\Gamma(G)$, where the vertices are labelled so that $d_i \leqs d_{i+1}$ for all $i$. By a theorem of Fulman and Guralnick \cite{FG17}, there is a positive constant $c_2$ such that  
$$d_1 \geqs c_2(m+1)$$
(the constant $c_2$ is independent of the choice of $G$).

\item[{\rm (c)}] \emph{P\'{o}sa's criterion}: $\Gamma(G)$ has a Hamiltonian cycle if $d_k \geqs k+1$ for all $k<m/2$ (see \cite[Exercise 10.21(b)]{Lov}).
\end{itemize}
If $d_i >m/2$ for all $i$ then P\'{o}sa's criterion immediately implies that $\Gamma(G)$ has a Hamiltonian cycle, so assume otherwise. Let $t$ be maximal such that $d_t <m/2$ and observe that
$$(m+1)^2 \cdot \mathbb{P}(G,2) = \sum_{i=1}^{m}d_i < \frac{1}{2}(m-1)t+(m-t)(m+1) < (m+1)^2 - \frac{1}{2}(m+1)t.$$
Therefore, applying (a) we get
$$1 - \frac{c_1}{m(G)} \leqs \mathbb{P}(G,2) \leqs 1 - \frac{t}{2(m+1)}$$
and thus
$$t \leqs \frac{2c_1(m+1)}{m(G)} \leqs c_2(m+1)-1$$
if $|G|$ is sufficiently large, where $c_2$ is the constant in (b). It follows that if  $1 \leqs k \leqs t$ then 
$$d_k \geqs d_{1} \geqs c_2(m+1) \geqs t+1 \geqs k+1.$$
Similarly, if $t+1 \leqs k < m/2$ then $d_k \geqs (m+1)/2 \geqs k+1$. Therefore, P\'{o}sa's criterion is satisfied and we conclude that $\Gamma(G)$ has a Hamiltonian cycle.
\end{proof}

Let $G$ be a $2$-generated finite group and let $N$ be a nontrivial normal subgroup of $G$. Observe that if $\Gamma(G)$ has no isolated vertices, then $G/N$ is cyclic (indeed, if $1 \ne x \in N$ and $G = \la x,y \ra$ for some $y \in G$, then $G/N = \la yN \ra$ is cyclic). The following conjecture is a combination (and strengthening) of conjectures in \cite{BGK, BGLMN}:

\begin{con}
Let $G$ be a finite group with $|G| \geqs 4$. Then the following are equivalent:
\begin{itemize}\addtolength{\itemsep}{0.2\baselineskip}
\item[{\rm (i)}] $G$ has spread $1$.
\item[{\rm (ii)}] $G$ has spread $2$.
\item[{\rm (iii)}] $\Gamma(G)$ has no isolated vertices.
\item[{\rm (iv)}] $\Gamma(G)$ is connected. 
\item[{\rm (v)}] $\Gamma(G)$ is connected with diameter at most $2$.
\item[{\rm (vi)}] $\Gamma(G)$ contains a Hamiltonian cycle.
\item[{\rm (vii)}] $G/N$ is cyclic for every nontrivial normal subgroup $N$.
\end{itemize}
\end{con}

\begin{rem}
Some comments on the status of this conjecture:
\begin{itemize}\addtolength{\itemsep}{0.2\baselineskip}
\item[{\rm (a)}] Note that any of the first six statements implies (vii). The following implications are also obvious:
$${\rm (i)} \iff {\rm (iii)},\; {\rm (ii)} \implies {\rm (v)},\; {\rm (vi)} \implies {\rm (iv)} \implies {\rm (iii)},$$
$${\rm (v)} \implies {\rm (iv)} \implies {\rm (iii)}$$ 
\item[{\rm (b)}] By \cite[Proposition 1.1]{BGLMN}, (vi) and (vii) are equivalent for soluble groups.
\item[{\rm (c)}] The conjectured equivalence of (i) and (vii) is \cite[Conjecture 1.8]{BGK}, and that of (vi) and (vii) is \cite[Conjecture 1.6]{BGLMN}.
\end{itemize}
\end{rem}

Notice that if the conjecture is true, then there is no finite group with spread $1$, but not spread $2$. 

Let us focus on the following weaker conjecture of Breuer, Guralnick and Kantor.

\begin{con}\label{c:sp}
Let $G$ be a finite group. Then $G$ has spread $1$ if and only if $G/N$ is cyclic for every nontrivial normal subgroup $N$ of $G$.
\end{con}

As noted above, Conjecture \ref{c:sp} has been verified for soluble groups. More importantly, Guralnick has recently established the following reduction theorem.

\begin{thm}
It is sufficient to prove Conjecture \ref{c:sp} for almost simple groups.
\end{thm}

In view of this result, we focus our attention on almost simple groups $G$ of the form $G = \la G_0, x \ra$, where $G_0$ is simple and $x \in {\rm Aut}(G_0)$. The goal is to prove that $G$ has spread $1$ (in fact, we aim for $s(G) \geqs 2$). This is work in progress:
\begin{itemize}\addtolength{\itemsep}{0.2\baselineskip}
\item $G_0 = {\rm Alt}(n)$ or sporadic: $s(G) \geqs 2$ by results in \cite{BGK}.
\item $G_0 = {\rm PSL}_{n}(q)$: $s(G) \geqs 2$ by the main theorem of \cite{BG}.
\item $G_0 \in \{{\rm PSp}_{n}(q), \O_n(q)\, \mbox{($nq$ odd)}\}$: $s(G) \geqs 2$ by work of Scott Harper (2016) in his PhD thesis (see \cite{Harper}).
\end{itemize}

The goal is to use Corollary \ref{c:mgen} to show that $u(G) \geqs 2$. To do this, we need to find a suitable element $y \in G_0x$ so that we can determine the maximal overgroups $\mathcal{M}(y)$ and estimate 
$$\sum_{H \in \mathcal{M}(y)}{\rm fpr}(x,G/H)$$
for all $x \in G$ of prime order (to get spread $1$, recall that we need to show that this sum is less than $1$). The set-up here is complicated by the fact that we have to choose $y$ in the coset $G_0x$, where $x$ is typically a field or graph automorphism of $G_0$. The following example illustrates some of the main ideas in the situation where $x$ is a field automorphism.

\begin{ex}[Harper]
Let $G_0 = {\rm Sp}_{n}(q)$, where $q=q_0^e$ is even, $e \geqs 5$ and $n=2m$ with $m \geqs 3$ odd. Let $X = {\rm Sp}_{n}(K)$ be the ambient simple algebraic group over the algebraic closure $K = \bar{\mathbb{F}}_{q}$ and let $\s:X \to X$ be a Frobenius morphism such that $X_{\s^e} = G_0$. Set $H_0 = X_{\s} = {\rm Sp}_{n}(q_0)<G_0$. 

Suppose $G = \la G_0,x\ra$, where $x$ is the restriction of $\s$ to $G_0$ (this is a field automorphism of order $e$). By the \emph{Lang-Steinberg theorem} (see \cite[Theorem 10.1]{St2}), for each $sx$ in the coset $G_0x$ there exists $a \in X$ such that $s=a^{-\sigma}a$. This allows us to define a map
\begin{equation}\label{e:shin}
f:\{\mbox{$G_0$-classes in $G_0x$}\} \to \{\mbox{$H_0$-classes in $H_0$}\}
\end{equation}
by sending $(sx)^{G_0}$ to $(a(sx)^ea^{-1})^{H_0}$. One checks that $f$ is well-defined and bijective (this map is sometimes called the \emph{Shintani correspondence}). One can also show that $f$ has nice fixed point properties for suitable actions of $G$ and $H_0$ (see \cite[Theorem 2.14]{BG}, for example).

The strategy is to choose an element $z \in H_0$ so that the maximal overgroups of $z$ in $H_0$ are somewhat restricted; hopefully this will allow us to control the maximal subgroups of $G$ containing a representative $y \in G_0x$ of the corresponding $G_0$-class in the coset $G_0x$. To do this, we take a semisimple element of the form $z =[A,B] \in H_0$, where $A \in {\rm Sp}_{2}(q_0)$ and $B \in {\rm Sp}_{n-2}(q_0)$ are irreducible (so $z$ preserves an orthogonal decomposition $V_0 = U \perp W$ of the natural module $V_0$ for $H_0$, with $\dim U=2$). Fix an element $y \in G_0x$ such that  $f(y^{G_0}) = z^{H_0}$. 

We need to determine the maximal subgroups of $G$ containing $y$. To do this, it is helpful to observe that $(|A|,|B|)=1$ so $\nu(y^{\ell})=2$ for some positive integer $\ell$. This quickly rules out maximal subgroups in the collections $\mathcal{C}_{3}$,  $\mathcal{C}_{4}$ and $\mathcal{C}_{7}$, and \cite[Theorem 7.1]{GurSax} (cf. Theorem \ref{t:gl}) can be used to further restrict the possibilities for $H$. By exploiting some additional properties of the bijection in \eqref{e:shin} one can show that there is a unique reducible subgroup in $\mathcal{M}(y)$ (of type ${\rm Sp}_{2}(q) \times {\rm Sp}_{n-2}(q)$) and also a unique $\mathcal{C}_{8}$-subgroup of type ${\rm O}_{n}^{\pm}(q)$. 

By carefully studying the subgroups in $\mathcal{C}_2 \cup \mathcal{C}_5$, and by applying the fixed point ratio bounds in Theorems \ref{t:liesax} and \ref{t:bur}, one can show that   
$$\sum_{H \in \mathcal{M}(y)}{\rm fpr}(x,G/H) < \a$$
for all $x \in G$ of prime order, where
$$\a = 2 \cdot \frac{4}{3q}+\left(1+2^{(m-1,e)}+(\log(e)+1)(q_0+1)(q_0^{m-1}+1)\right) \cdot \frac{(4q+4)^{1/2-1/2m}}{q^{m-1}}$$
Now $q \geqs 32$ since $q=q_0^e$ with $e \geqs 5$, and it is straightforward to show that $\a<1/3$. Therefore $u(G) \geqs 3$ by Corollary \ref{c:mgen}.
\end{ex} 

\section{Monodromy groups}\label{s:mono}

In this section we turn to an application of fixed point ratios in the study of coverings of Riemann surfaces, focussing on the solution to Problem B stated in the introduction.

\subsection{Preliminaries}

Let $X$ be a compact connected Riemann surface of genus $g \geqs 0$ and let $Y=\mathbb{P}^1(\mathbb{C})$ be the Riemann sphere. Let $f:X \to Y$ be a branched covering of degree $n$.
This means that $f$ is a meromorphic function with a finite set of branch points $B = \{y_1, \ldots, y_k\} \subset Y$ (with $k \geqs 2$) such that the restriction of $f$ to $X^0 = X \setminus f^{-1}(B)$ is a covering map of degree $n$ (that is, $|f^{-1}(y)|=n$ for all $y \in Y^0=Y \setminus B$, so generically $f$ is an ``$n$-to-$1$" mapping). 

Fix $y_0 \in Y^0$ and let $\O = f^{-1}(y_0) = \{x_1, \ldots, x_n\}$ be the fibre of $y_0$. Let $\gamma$ be a loop in $Y^0$ based at $y_0$. For each $x_i \in \O$, we can lift $\gamma$ via $f$ to a path $\tilde{\gamma}_{i}$ in $X$ beginning at $x_i$. The endpoint $\tilde{\gamma}_{i}(1)$ is also in $\O$ and the corresponding map $\s_{\gamma}: x_i \mapsto \tilde{\gamma}_{i}(1)$ is a permutation  of $\O$, which is independent of the homotopy type of $\gamma$. In this way, we obtain a homomorphism 
$$\varphi: \pi_{1}(Y^0,y_0) \to {\rm Sym}(\O),\;\; [\gamma] \mapsto \s_{\gamma}$$
from the fundamental group of $Y^0$ with base point $y_0$.
The image of this map is a permutation group of degree $n$. Moreover, the path connectedness of $Y^0$ implies that the group we obtain in this way is independent of the choice of base point $y_0 \in Y^0$, up to permutation isomorphism. This allows us to make the following definition.

\begin{defn}
The \emph{monodromy group} ${\rm Mon}(X,f)$ of $f$ is defined to be the image of $\varphi$. 
\end{defn}

Since $X$ has genus $g$, we refer to ${\rm Mon}(X,f)$ as a \emph{monodromy group of genus $g$}. The connectedness of $X$ implies that ${\rm Mon}(X,f)$ is a transitive permutation group.

\begin{ex}
Let $n \geqs 2$ be an integer and consider the map $f:X \to Y$ given by $f(z)=z^n$, where $X = Y=\mathbb{P}^{1}(\mathbb{C})$. This is a branched covering
of degree $n$ with branch points $B=\{0,\infty\}$ and $\O = f^{-1}(1) = \{e^{2\pi i k/n}\,:\, k = 0, 1, \ldots, n-1\}$. Here the monodromy group is cyclic of order $n$, generated by the permutation $\a \mapsto \zeta\alpha$ of $\O$, with $\zeta = e^{2\pi i/n}$. 
\end{ex}

There is a natural generating set $\{\gamma_1, \ldots, \gamma_k\}$ for $\pi_{1}(Y^0,y_0)$, where $\gamma_i$ is a loop that encircles the $i$-th branch point $y_i$ (and no other branch point) with the property that the $\gamma_i$ only meet at $y_0$. Moreover, by relabelling if necessary, one can show that the product $\gamma_1 \cdots \gamma_k$ is homotopy equivalent to the trivial loop based at $y_0$, so $\gamma_1 \cdots \gamma_k=1$ and by a theorem of Hurwitz (1891) we have
$$\pi_{1}(Y^0,y_0) = \la \gamma_1, \ldots, \gamma_k \mid \gamma_1 \cdots \gamma_k=1 \ra.$$
This implies that ${\rm Mon}(X,f) = \la \s_1, \ldots, \s_k\ra$ and $\s_1\cdots\s_k=1$, where $\s_i = \s_{\gamma_i}$ as above. 

\begin{quest}
Which transitive permutation groups $G \leqs {\rm Sym}(\O)$ of degree $n$ occur as the monodromy group of a branched covering $f:X \to \mathbb{P}^1(\mathbb{C})$ of genus $g$ and degree $n$? 
\end{quest}

A necessary and sufficient condition is provided by the \emph{Riemann Existence Theorem} below (see \cite{Vol} for a modern treatment). In the statement, recall that ${\rm ind}(x) = n-t$ is the \emph{index} of a permutation $x \in {\rm Sym}(\O)$, where $t$ is the number of cycles of $x$ on $\O$. Equivalently, ${\rm ind}(x)$ is the minimal $\ell$ such that $x$ is a product of $\ell$ transpositions. Note that if  $x_1, \ldots, x_k$ are permutations of $\O$ with $x_1 \cdots x_k = 1$, then $\sum_{i}{\rm ind}(x_i)$ is even.

\begin{thm}[Riemann Existence Theorem]\label{t:ret}
Let $G \leqs {\rm Sym}(\O)$ be a transitive group of degree $n$. Then $G$ is isomorphic to a monodromy group ${\rm Mon}(X,f)$ for some compact connected Riemann surface $X$ of genus $g$ and branched covering $f:X \to \mathbb{P}^1(\mathbb{C})$ if and only if $G$ has a generating set $\{g_1, \ldots, g_k\}$ such that $g_1 \cdots g_k=1$ and 
\begin{equation}\label{e:gen}
\sum_{i=1}^{k}{\rm ind}(g_i) = 2(n+g-1).
\end{equation}
\end{thm}

This fundamental result allows us to translate questions about monodromy groups to purely group-theoretic problems concerning finite permutation groups. One of the main problems is to understand the structure of monodromy groups of genus $g$, specifically in terms of the composition factors of such groups. This is formalised in a highly influential conjecture of Guralnick and Thompson from 1990 \cite{GT}, which we will discuss below. 

\begin{rem}\label{r:mon}
There is a well understood connection between branched covers of $\mathbb{P}^1(\mathbb{C})$ and finite extensions of the field $\mathbb{C}(t)$. More precisely, if $f:X \to \mathbb{P}^1(\mathbb{C})$ is a branched covering then $\mathbb{C}(X)/\mathbb{C}(t)$ is a finite extension, where $\mathbb{C}(X)$ and $\mathbb{C}(t)$ denote the function fields of $X$ and $\mathbb{P}^1(\mathbb{C})$, respectively. It turns out that the Galois group of the normal closure of this extension is the monodromy group ${\rm Mon}(X,f)$. Therefore, Theorem \ref{t:ret} can be interpreted in terms of the inverse Galois problem. In particular, this viewpoint permits natural generalisations in which $\mathbb{C}$ is replaced by some other algebraically closed field (of any characteristic). We will return to this more general set-up at the end of the section.
\end{rem}

\subsection{The Guralnick-Thompson conjecture}

Motivated by the discussion above, we define the \emph{genus} of a finite group as follows.

\begin{defn} 
Let $G \leqs {\rm Sym}(\O)$ be a finite transitive permutation group and let $E = \{g_1, \ldots, g_k\}$ be a generating set for $G$ with $g_1 \cdots g_k=1$. Define the genus $g = g(G,\O,E)$ as in \eqref{e:gen} and define $g(G,\O)$ to be the minimal value of $g(G,\O,E)$ over all such generating sets $E$ (for any $k$). We say that a finite group $G$ has \emph{genus} $g$ if it 
has a faithful transitive $G$-set $\O$ such that $g(G,\O) \leqs g$. 
\end{defn}

\begin{ex}
Let $G = \la g_1 \ra$ be a cyclic group of order $n$. Set $E = \{g_1, g_1^{-1}\}$ and consider the regular action of $G$ on itself. Then ${\rm ind}(g_1) = {\rm ind}(g_1^{-1}) = n-1$, so \eqref{e:gen} implies that $G$ has genus zero.
\end{ex}

\begin{ex}
Let $G = {\rm Sym}(n) = \la g_1, g_2, g_2^{-1}g_1^{-1} \ra$, where $g_1 = (1,2)$ and $g_2 = (1,2, \ldots, n)$, and consider the natural action of $G$ of degree $n$. The indices of the respective generators are $1$, $n-1$ and $n-2$, hence $G$ has genus zero. Similarly, every alternating group has genus $0$.
\end{ex}

Fix a non-negative integer $g$ and let $\mathcal{C}(g)$ be the set of composition factors of groups of genus $g$. Note that $\mathcal{C}(0) \subseteq \mathcal{C}(g)$. In view of the above examples, it follows that $\mathcal{C}(g)$ contains every simple cyclic and alternating group. Therefore we focus on $\mathcal{E}(g)$, which is the set of nonabelian, non-alternating composition factors of groups of genus $g$. 

The following theorem establishes a conjecture of Guralnick and Thompson \cite{GT}, which we stated as Problem B in the introduction. 

\begin{thm}
$\mathcal{E}(g)$ is finite for each $g$.
\end{thm}

\begin{proof}[Sketch proof]
We briefly sketch the main steps, highlighting the central role played by fixed point ratios. 

\vs

\noindent \emph{Step 1.} Reduction to almost simple primitive groups.

\vs

By work of Aschbacher, Guralnick, Neubauer, Thompson and others \cite{A90,G92, GT, Neu}, it is sufficient to show that there are only finitely many primitive almost simple groups of Lie type of genus $g$. The key ingredient in this highly nontrivial reduction is the Aschbacher-O'Nan-Scott theorem (as formulated in \cite{AS}) on the structure of finite groups with a core-free maximal subgroup. See \cite[Section 5]{G92} for more details.

\vs

\noindent \emph{Step 2.} Fixed point ratio estimates.

\vs

The next result provides the connection to fixed point ratios (see \cite[Corollary 2]{G92}).

\begin{prop}\label{p:keyy}
Let $\mathcal{X}$ be the set of nonabelian, non-alternating finite simple groups such that  
\begin{equation}\label{e:85}
\max_{1 \ne x \in G} {\rm fpr}(x,\O) \leqs \frac{1}{86}
\end{equation}
for every almost simple primitive group $G \leqs {\rm Sym}(\O)$ with socle $G_0 \in \mathcal{X}$. Then $\mathcal{X} \cap \mathcal{E}(g)$ is finite for every non-negative integer $g$. 
\end{prop}

\begin{proof}
Let $G \leqs {\rm Sym}(\O)$ be a primitive almost simple group of degree $n$ with socle $G_0 \in \mathcal{X}$, so \eqref{e:85} holds. Fix a generating set $\{g_1, \ldots, g_k\}$ for $G$ such that $g_1 \cdots g_k=1$ and define $g$ as in \eqref{e:gen}. Note that $k \geqs 3$. Let $d_i$ be the order of $g_i$ and let ${\rm orb}(g_i)$ be the number of cycles of $g_i$ on $\O$. Without loss of generality, we may assume that $d_i \leqs d_{i+1}$ for all $i$.

By the orbit-counting lemma we have  
$${\rm orb}(g_i) = \frac{n}{d_i} \sum_{x \in \la g_i \ra}{\rm fpr}(x,\O) = \frac{n}{d_i}\left(1+\sum_{1 \ne x \in \la g_i \ra}{\rm fpr}(x,\O)\right)$$
and thus
$$\sum_{i=1}^{k}{\rm ind}(g_i) = \sum_{i=1}^{k}(n - {\rm orb}(g_i)) \geqs \frac{85}{86}n\sum_{i=1}^{k}\frac{d_i-1}{d_i}.$$
Since $G$ is insoluble and $G \not\cong {\rm Alt}(5)$, \cite[Proposition 2.4]{GT} implies that 
\begin{equation}\label{e:zar}
\sum_{i=1}^{k}\frac{d_i-1}{d_i} \geqs \frac{85}{42}.
\end{equation}
(Note that equality holds if and only if $k=3$ and $(d_1,d_2,d_3)=(2,3,7)$, so $G$ is a \emph{Hurwitz group}, such as ${\rm PSL}_{2}(7)$.) Therefore
$$2(n+g-1) = \sum_{i=1}^{k}{\rm ind}(g_i) \geqs \frac{85}{86}\cdot \frac{85}{42}n$$
and thus 
$n \leqs 7224(g-1)$. The result follows.
\end{proof}

\noindent \emph{Step 3.} Bounded rank. 

\vs

We combine Proposition \ref{p:keyy} with Theorem \ref{t:liesax}, noting that $4/3q \leqs 1/86$ when $q > 113$ (the almost simple groups with socle ${\rm PSL}_{2}(q)$ excluded in Theorem \ref{t:liesax} can be handled separately). 

\vs

Therefore, to complete the proof of the theorem we may assume that $G$ is an almost simple classical group of large rank (in other words, we may assume that the dimension of the natural module of the socle of $G$ is arbitrarily large).

\vs

\noindent \emph{Step 4.} Classical groups in non-subspace actions.

\vs

Let $G \leqs {\rm Sym}(\O)$ be an almost simple primitive group over $\mathbb{F}_{q}$ with socle $G_0$ in a non-subspace action (see Definition \ref{d:sub}). Let $m$ be the dimension of the natural module for $G_0$. By Theorem \ref{t:lsh} there is a constant $\delta>0$ such that 
$$\max_{1 \ne x \in G}{\rm fpr}(x)<q^{-\delta m},$$ 
so there are at most finitely many groups $G$ with $\max_{1 \ne x \in G}{\rm fpr}(x) > 1/86$. Now apply Proposition \ref{p:keyy}.

\vs

\noindent \emph{Step 5.} Classical groups in subspace actions.

\vs

To complete the proof it remains to handle the subspace actions of classical groups. It is sufficient to prove the following result (see \cite{FM_annals}).

\begin{prop}
Fix a prime power $q$ and non-negative integer $g$. Then there exists a constant $N = N(q)$ such that if $G \leqs {\rm Sym}(\O)$ is any primitive almost simple classical group over $\mathbb{F}_{q}$ in a subspace action of genus $g$ then either $m \leqs N$ or $n \leqs 2000g$, where $m$ is the dimension of the natural module and $n$ is the degree of $G$.
\end{prop}

In other words, if the dimension of the natural module $V$ is large enough, then the degree of $G$ is bounded above by a (linear) function of $g$ and the result follows. The key tool is Theorem \ref{t:fm_sub} -- the details are rather technical, so we only provide a rough outline of the argument.

Let $G = \la g_1, \ldots, g_k\ra$ be a generating set such that $g_1 \cdots g_k=1$ and \eqref{e:gen} holds. Let $d_i$ be the order of $g_i$ and label the $g_i$ so that $d_i \leqs d_{i+1}$ for all $i$. Let $\varphi$ be Euler's function and set 
\begin{align*}
\nu(d) & = \min\{\nu(x) \,:\, x \in G,\, |x|=d\}, \\
\a_0(d) &  = 1-\frac{1}{d}\sum_{a|d}\varphi(a)q^{-\nu(a)}
\end{align*}
for each natural number $d$. Note that $\a_0(d) \geqs 1/4$ if $d \geqs 2$. 
Set 
$$\a(g_i) = \frac{{\rm ind}(g_i)}{n},\;\; \s = \sum_{i=1}^{k}\a(g_i)$$ 
and observe that it suffices to show that $\s>2.001$ if $m$ is sufficiently large (where $m$ is the dimension of the natural module).

We may assume that $m$ is large enough so that Theorem \ref{t:fm_sub} gives 
$${\rm fpr}(x)<q^{-\nu(x)}+10^{-3}$$ 
for all $1 \ne x \in G \cap {\rm PGL}(V)$. Then 
\begin{align*}
\a(g_i) = 1-\frac{1}{d_i}\sum_{j=1}^{d_i}{\rm fpr}(g_i^j) & > 1-\frac{1}{d_i}\left(\sum_{j=1}^{d_i} q^{-\nu(|g_i^j|)}\right) - 10^{-3} \\
& = 1-\frac{1}{d_i}\left(\sum_{a|d_i} \varphi(a)q^{-\nu(a)}\right) - 10^{-3} \\
& = \a_0(d_i) - 10^{-3} 
\end{align*}
which is at least $0.249$, whence $\s \geqs 0.249k$. Therefore, we may assume that $k \leqs 8$. With further work it is possible to reduce to the minimal case $k=3$, where the final analysis splits into several subcases according to the values of $d_1, d_2$ and $d_3$. It is worth noting that an important tool in the latter reduction is the fact that 
$$\sum_{i=1}^{k}\nu(g_i) \geqs 2m,$$
which follows from a well known theorem of Scott \cite{Scott}. 
\end{proof}

\subsection{Genus zero groups}

It remains an open problem to explicitly determine the simple groups in $\mathcal{E}(g)$, although there has been some significant recent progress in the low genus cases, and in particular the special case $g=0$. For example, the sporadic groups in $\mathcal{E}(0)$ have been determined by Magaard \cite{Mag}; the examples are as follows:
$${\rm M}_{11},\, {\rm M}_{12},\, {\rm M}_{22},\, {\rm M}_{23},\, {\rm M}_{24},\, {\rm J}_{1}, \, {\rm J}_{2},\, {\rm Co}_{3}, \, {\rm HS}.$$
By work of Frohardt and Magaard \cite{FM_3, FM_2}, the only exceptional groups in $\mathcal{E}(0)$ are ${}^2B_2(8)$, $G_2(2)' \cong {\rm PSU}_{3}(3)$ and ${}^2G_2(3)' \cong {\rm PSL}_{2}(8)$. The relevant groups of the form ${\rm PSL}_{2}(q)$ and ${\rm PSU}_{3}(q)$ are determined in \cite{FGM2}:
\begin{align*}
{\rm PSL}_{2}(q): & \;\; q \in \{7,8,11,13,16,17,19,25,27,29,31,32,37,41,43,64\} \\
{\rm PSU}_{3}(q): & \;\; q \in \{3,4,5\}
\end{align*}

There is work in progress by Frohardt, Guralnick, Hoffman and Magaard to extend the results in \cite{FGM2} to higher rank classical groups, leading to a complete classification of all primitive permutation groups of genus zero. In fact, the ultimate aim is to determine the primitive groups of genus at most two, building on earlier work in \cite{FGM}. It is anticipated that these results will have interesting number-theoretic applications. 

At a conference in July 2016 (\emph{Algebraic Combinatorics and Group Actions}, Herstmonceux Castle, UK), Frohardt announced that the groups in $\mathcal{E}(0)$ have been determined. The complete list is as follows:
$$\begin{array}{l}
{\rm PSL}_{2}(q),\; 7 \leqs q \leqs 43,\; q \ne 9,23 \\
{\rm PSL}_{2}(64) \\
{\rm PSL}_{3}(q),\, q \in \{3,4,5,7\} \\
{\rm PSL}_{4}(3),\, {\rm PSL}_{4}(4), \, {\rm PSL}_{5}(2),\, {\rm PSL}_{5}(3),\;{\rm PSL}_{6}(2) \\
{\rm PSU}_{3}(3),\,  {\rm PSU}_{3}(4),\,  {\rm PSU}_{3}(5) \\
{\rm PSp}_{4}(3),\, {\rm PSp}_{4}(4),\, {\rm PSp}_{4}(5) \\
{\rm PSp}_{6}(2),\, {\rm PSp}_{8}(2) \\
{\rm P\O}_{8}^{+}(2),\, {\rm P\O}_{8}^{-}(2) \\
{}^2B_2(8) \\
{\rm M}_{11},\, {\rm M}_{12},\, {\rm M}_{22},\, {\rm M}_{23},\, {\rm M}_{24},\, {\rm J}_{1}, \, {\rm J}_{2},\, {\rm Co}_{3}, \, {\rm HS}
\end{array}$$

\begin{ex}
Notice that $G={\rm PSL}_{2}(23) \not\in \mathcal{E}(0)$. To see that $G$ does not have a primitive genus zero action, first observe that $G$ has five conjugacy classes of maximal subgroups, represented by 
$$Z_{23}{:}Z_{11}, \;\; D_{24}, \;\; D_{22},\;\; {\rm Sym}(4) \mbox{ (two classes)}.$$ 
Fix a maximal subgroup $M$ and set $n=|G:M|$. As recorded in the following table, it is straightforward to compute ${\rm ind}(x)$ for each non-identity $x \in G$ (with respect to the action of $G$ on $G/M$):
$$\begin{array}{|lcc|ccccccc|} \hline
M & n & 2n-2 & |x|=2 & 3 & 4 & 6 & 11 & 12 & 23 \\ \hline
Z_{23}{:}Z_{11} & 24 & 46 & 12 & 16 & 18 & 20 & 20 & 22 & 22 \\
D_{24} & 253 & 504 & 120 & 168 & 186 & 208 & 230 & 230 & 242 \\
D_{22} & 276 & 550 & 132 & 184 & 208 & 238 & 250 & 252 & 264 \\
{\rm Sym}(4) & 253 & 504 & 122 & 166 & 186 & 208 & 230 & 230 & 242 \\ \hline
\end{array}$$

Seeking a contradiction, suppose that $G = \la g_1, \ldots, g_k\ra$ with $g_1 \cdots g_k=1$ and $\sum_{i}{\rm ind}(g_i) = 2n-2$. As before, let $d_i$ denote the order of $g_i$ and assume $d_{i} \leqs d_{i+1}$ for all $i$. 

Suppose $M$ is the Borel subgroup $Z_{23}{:}Z_{11}$, so $2n-2=46$. From the above table, it follows that $k=3$, $d_1=2$ and $d_2 \in \{2,3\}$. If $d_2=2$ then $\sum_{i}(d_i-1)/d_i < 2$, which contradicts the bound in \eqref{e:zar}. Therefore, $d_2=3$  and a second application of \eqref{e:zar} forces $d_3 \geqs 11$. But this implies that 
$$\sum_{i=1}^{k}{\rm ind}(g_i) \geqs 12+16+20>46,$$ 
which is a contradiction. The other cases are similar.
\end{ex}

\subsection{Generalisations}

As noted in Remark \ref{r:mon}, there are natural generalisations to other fields. Let $k$ be an algebraically closed field of characteristic $p \geqs 0$ and let $f:X \to Y$ be a finite separable cover of smooth projective curves over $k$. Let $G$ be the corresponding monodromy group, which is the Galois group of the normal closure of the extension $k(X)/k(Y)$ of function fields. As for $p=0$, in positive characteristic we can seek restrictions on the structure of $G$ according to the genus of $X$. There is still a translation of the problem to group theory, but the set-up is much more complicated. For example, there is no known analogue of Riemann's Existence Theorem and further complications arise if the given cover $f$ is wildly ramified (the proof of the Guralnick-Thompson conjecture goes through essentially unchanged in the tamely ramified case). 

The following conjecture of Guralnick is the positive characteristic analogue of the Guralnick-Thompson conjecture  (see \cite[Conjecture 1.6]{G03}). In order to state the conjecture, let $p$ be a prime and let $S$ be a nonabelian simple group. We say that $S$ has genus $g$ (in characteristic $p$) if $S$ is a composition factor of the monodromy group of a finite separable cover $f:X\to Y$ of smooth projective curves over an algebraically closed field of characteristic $p$ with $X$ of genus at most $g$.

\begin{con}
Let $g \geqs 0$ be an integer and let ${\rm E}_{p}(g)$ be the set of nonabelian non-alternating simple groups of genus $g$ in characteristic $p>0$. Then
$${\rm E}_{p}(g) \cap \left(\bigcup_{r \ne p}{{\rm Lie}(r)}\right)$$
is finite, where ${\rm Lie}(r)$ is the set of finite simple groups of Lie type in characteristic $r$.
\end{con}

The condition $r \ne p$ in the conjecture is necessary. For example, work of Abhyankar (see \cite{Abh} and the references therein) shows that ${\rm E}_{p}(0)$ contains every simple classical group in characteristic $p$. Guralnick's conjecture is still open in its full generality, and we refer the reader to the survey article \cite{G03} for further details. Here it is worth highlighting \cite[Theorem 1.5]{G03}, which shows that the conjecture holds if we replace ${\rm Lie}(r)$ by the set of finite simple groups of Lie type in characteristic $r$ of bounded dimension.

\section{Bases}\label{s:bases}

In this final section we introduce the classical notion of a base for a permutation group and we discuss how probabilistic methods, based on fixed point ratio estimates, have been used to establish strong results on the minimal size of  bases for simple groups. In particular, we will sketch a solution to Problem C in the introduction.

\subsection{Preliminaries}\label{s:bas_intro}

We begin by defining the base size of a permutation group.

\begin{defn}\label{d:base}
Let $G \leqs {\rm Sym}(\O)$ be a permutation group.  A subset $B$ of $\O$ is a \emph{base} for $G$ if $\bigcap_{\a \in B}G_{\a} = 1$. The \emph{base size} of $G$, denoted by $b(G)$, is the minimal cardinality of a base.
\end{defn}

\begin{exs}\label{ex:base}
\mbox{ }
\begin{itemize}\addtolength{\itemsep}{0.2\baselineskip}
\item[1.] $b(G)=1$ if and only if $G$ has a regular orbit on $\O$.  
\item[2.] $b(G)=n-1$ for the natural action of $G = {\rm Sym}(n)$ on $\O = \{1, \ldots, n\}$. 
\item[3.] $b(G) = \dim V$ for the natural action of $G = {\rm GL}(V)$ on $\O=V$. 
\item[4.] $b(G) = \dim V +1$ for the action of $G = {\rm PGL}(V)$ on the set of $1$-dimensional subspaces of $V$.
\end{itemize}
\end{exs}

\begin{rem}
Bases arise naturally in several different contexts:

\vs

\noindent a. \emph{Abstract group theory.}
Let $G$ be a finite group and let $H$ be a core-free subgroup, so we may view $G$ as a permutation group on $\O=G/H$. In this context, $b(G)$ is the size of the smallest subset $S \subseteq G$ such that $\bigcap_{x \in S}H^x = 1$.

\vs

\noindent b. \emph{Permutation group theory.}
Let $G$ be a permutation group of degree $n$ and let $B$ be a base for $G$. If 
$x,y \in G$ then
$$\a^x = \a^y \mbox{ for all $\a \in B$} \iff xy^{-1} \in \bigcap_{\a \in B}G_{\a} \iff x = y$$
and thus $|G| \leqs n^{|B|}$. In this way, (upper) bounds on $b(G)$ can be used to bound the order of $G$. 

\vs

\noindent c. \emph{Computational group theory.} The concept of a \emph{base and strong generating set} was introduced by Sims \cite{Sims} in the early 1970s, and it plays a fundamental role in the computational study of finite permutation groups (e.g. for computing the order of the group, and testing membership). See \cite[Section 4]{Seress_book} for more details.

\vs

\noindent d. \emph{Graph theory.} 
Let $\Gamma$ be a graph with vertices $V$ and automorphism group $G = {\rm Aut}(\Gamma) \leqs {\rm Sym}(V)$.  Then
\begin{align*}
b(G) & = \mbox{the \emph{fixing number} of $\Gamma$} \\
& = \mbox{the \emph{determining number} of $\Gamma$} \\
& = \mbox{the \emph{rigidity index} of $\Gamma$}
\end{align*}
is a well-studied graph invariant. See the survey article by Bailey and Cameron \cite{BCam} for further details.
\end{rem}

Let $G$ be a permutation group of degree $n$. In general, it is very difficult to compute $b(G)$ precisely (indeed, algorithmically, this is known to be an \emph{NP-hard} problem; see \cite{Blaha}), so we focus on bounds, and in particular upper bounds in view of applications. It is easy to see that 
\begin{equation}\label{e:bou}
\frac{\log |G|}{\log n} \leqs b(G) \leqs \log_2 |G|
\end{equation} 
and it is straightforward to construct transitive groups $G$ such that $b(G)$ is at either end of this range. A well known conjecture of Pyber \cite{Pyber} from the early 1990s asserts that the situation for primitive groups is rather more restrictive, in the sense that there is an absolute constant $c$ such that 
$$b(G) \leqs c\frac{\log |G|}{\log n}$$ 
for any primitive group $G$ of degree $n$. This conjecture has very recently been proved by Duyan, Halasi and Mar\'{o}ti \cite{DHM}, building on the earlier work of many authors.

The following lemma reveals a connection between the base size and the minimal degree of a transitive group (cf. Section \ref{ss:md}).

\begin{lem}\label{l:bgmu}
Let $G \leqs {\rm Sym}(\O)$ be a transitive group of degree $n$. Then $b(G)\mu(G) \geqs n$.
\end{lem}

\begin{proof}
Let $B$ be a base of minimal size and let $S$ be the support of an element $1 \neq g \in G$ of minimal degree. If $B^x \cap S = \emptyset$ for some $x \in G$, then $B^x \subseteq \O \setminus S$, so $g$ fixes every element of $B^x$, but this is not possible since $B^x$ is a base. Therefore $|B^x \cap S| \geqs 1$ for all $x \in G$. 

Next we claim that $|\{x \in G \,:\, \a \in B^x\}|=|B||G|/n$ for all $\a \in \O$. Consider $\a_1 \in B$. Fix $y \in G$ such that $\a^y = \a_1$. Then
$$\{x \in G \,:\, \a=\a_1^x\} = \{x \in G \,:\, \a=\a^{yx}\} = \{x \in G \,:\, yx \in G_{\a} \} =  y^{-1}G_{\a},$$
so
$|\{x \in G \,:\, \a \in B^x\}|= |B||G_{\a}| = |B||G|/n$ as claimed. We conclude that 
$$|G| \leqs \sum_{x \in G}|B^x \cap S| = \sum_{\a \in S}|\{x \in G \,:\, \a \in B^x\}| = |S||B||G|/n$$
and thus $\mu(G)b(G) = |S||B| \geqs n$.
\end{proof}

Let $G \leqs {\rm Sym}(\O)$ be a primitive group of degree $n$. Since 
$$b({\rm Sym}(n))=n-1,\;\; b({\rm Alt}(n))=n-2$$ 
we will assume that $G \neq {\rm Alt}(n), {\rm Sym}(n)$. Determining upper bounds on $b(G)$ in terms of $n$ is an old problem. We record some results:
\begin{itemize}\addtolength{\itemsep}{0.2\baselineskip}
\item Bochert \cite{Boch}, 1889: $b(G) \leqs n/2$
\item Babai \cite{Babai2}, 1981: $b(G) \leqs c \sqrt{n}\log n$ for some constant $c$ (independent of CFSG)
\item Liebeck \cite{L10}, 1984: $b(G) \leqs c\sqrt{n}$ (using CFSG)
\end{itemize}

\begin{rem}
It is easy to see that Liebeck's bound is best possible. For example, if 
$G={\rm Sym}(m)$ and $\O$ is the set of $2$-element subsets of $\{1,\ldots ,m\}$, 
then $n=\binom{m}{2}$ and $b(G) \sim \frac{2}{3}m =O(\sqrt{n})$. For instance, if $m \equiv 1 \imod{12}$ then 
$$\{\{1,2\},\{2,3\}, \;\; \{4,5\},\{5,6\}, \;\; \ldots, \;\;  \{m-3,m-2\}, \{m-2,m-1\}\}$$
is a base of size $2(m-1)/3$ (this is optimal).
\end{rem}

Stronger bounds are attainable if we focus on specific families of primitive groups. For instance, a striking theorem of Seress \cite{Seress} states that $b(G) \leqs 4$ if $G$ is soluble. For the remainder we will focus on almost simple primitive groups.

\subsection{Simple groups and probabilistic methods}

Let $G \leqs {\rm Sym}(\O)$ be a primitive almost simple group of degree $n$, with socle $G_0$ and point stabiliser $H$. In studying the base size of such groups, it is natural to make a distinction between \emph{standard} and \emph{non-standard} groups, according to the following definition (see 
Definition  \ref{d:sub} for the notion of a subspace action of a classical group).

\begin{defn}\label{d:std}
We say that $G$ is \emph{standard} if one of the following holds:
\begin{itemize}\addtolength{\itemsep}{0.2\baselineskip}
\item[(i)] $G_0={\rm Alt}(m)$ and $\O$ is an orbit of subsets or partitions of $\{1, \ldots ,m\}$; 
\item[(ii)] $G$ is a classical group in a subspace action.
\end{itemize}
\noindent Otherwise, $G$ is \emph{non-standard}. (Note that we will only use the terms \emph{standard} and \emph{non-standard} in the context of a primitive group.)
\end{defn}

In general, if $G$ is standard then $H$ is ``large" in the sense that $|G|$ is not bounded above by a polynomial in $n = |G:H|$ of fixed degree. For example, if we take the standard action of $G={\rm PGL}_{m}(q)$ on $1$-spaces then $|G| \sim q^{m^{2}-1}$ and $n \sim q^{m-1}$. In view of \eqref{e:bou}, this implies that the base size of such a standard group can be arbitrarily large (indeed, we already noted that $b(G)=m+1$ for the given action of ${\rm PGL}_{m}(q)$). 

Now assume $G$ is non-standard. By a theorem of Cameron, there is an absolute constant $c$ such that $|G| \leqs n^c$ for any such group $G$. In  later work, Liebeck \cite{L10} showed that $c=9$ is sufficient, and this was extended by Liebeck and Saxl \cite{LS44} to give the following. 

\begin{thm}
Let $G$ be a non-standard group of degree $n$. Then either $|G| \leqs n^5$, or $(G,n) = ({\rm M}_{23},23)$ or $({\rm M}_{24},24)$.
\end{thm}

This result suggests that non-standard groups may admit small bases. Indeed, the following striking theorem of Liebeck and Shalev \cite{LSh2} shows that this is true in a very strong sense. The proof uses probabilistic methods based on fixed point ratio estimates.

\begin{thm}\label{t:bls}
There is an absolute constant $c$ such that if $G \leqs {\rm Sym}(\O)$ is a non-standard permutation group then the probability that a randomly chosen $c$-tuple in $\O$ is a base for $G$ tends to $1$ as $|G|$ tends to infinity.
\end{thm}

\begin{rem}
This asymptotic result was conjectured by Cameron and Kantor \cite{CK}, and they showed that it holds for alternating and symmetric groups with the best possible constant $c=2$. 
\end{rem}

Let us explain the connection between fixed point ratios and base sizes. Let $c$ be a positive integer and let $Q(G,c)$ be the probability that a randomly chosen $c$-tuple of points in $\O$ is \emph{not} a base for $G$, so 
$$b(G) \leqs c \iff Q(G,c)<1.$$ 
Observe that a $c$-tuple in $\O$ fails to be a base if and only if it is fixed by an element $x \in G$ of prime order, and note that the probability that a randomly chosen $c$-tuple is fixed by $x$ is equal to ${\rm fpr}(x)^{c}$. Let $\mathcal{P}$ be the set of elements of prime order in $G$, and let $x_{1}, \ldots, x_{k}$ represent the $G$-classes in $\mathcal{P}$. Then
\begin{equation}\label{e:bd}
Q(G,c) \leqs \sum_{x \in \mathcal{P}}{{\rm fpr}(x)^{c}}=\sum_{i=1}^{k}{|x_{i}^{G}|\cdot {\rm fpr}(x_{i})^{c}}=:\what{Q}(G,c).
\end{equation}
We have thus established the following key lemma, which allows us to exploit upper bounds on fixed point ratios to bound the base size.

\begin{lem}
If $\what{Q}(G,c)<1$ then $b(G) \leqs c$.
\end{lem}

\begin{proof}[Sketch proof of Theorem \ref{t:bls}]
Let $G_0$ denote the socle of $G$. In view of the work of Cameron and Kantor in \cite{CK}, we may assume that $G_0$ is a group of Lie type over $\mathbb{F}_{q}$. 

First we claim that $b(G) \leqs 500$ if $G_0$ is an exceptional group. To see this, we apply Theorem \ref{t:liesax}, which states that 
$${\rm fpr}(x) \leqs \frac{4}{3q}$$ 
for all non-identity elements $x \in G$. Now $|G| \leqs |{\rm Aut}(E_8(q))|<q^{249}$, so 
$$Q(G,500) \leqs \what{Q}(G,500) \leqs \left(\frac{4}{3q}\right)^{500}\sum_{i=1}^{k}{|x_{i}^{G}|}<\left(\frac{4}{3q}\right)^{500}|G|$$
which is at most $q^{-1}$ since $|G|<q^{249}$. The claim follows. Similarly, if $G_0$ is a non-standard classical group of rank $r$, then the same argument yields $b(G) \leqs c(r)$ (with $Q(G,c(r)) \to 0$ as $q$ tends to infinity).

The key tool to handle the remaining non-standard classical groups of arbitrarily large rank is Theorem \ref{t:lsh}, which states that  there is a constant $\e>0$ such that 
\begin{equation}\label{e:starr}
{\rm fpr}(x)<|x^{G}|^{-\e}
\end{equation}
for all $x \in G$ of prime order (recall that the non-standard hypothesis is essential). Let $m$ be the dimension of the natural module for $G_0$. We need two facts:   
\begin{enumerate}\addtolength{\itemsep}{0.2\baselineskip}
\item[1.] $G$ has at most $q^{4m}$ conjugacy classes of elements of prime order (for example, this follows from \cite[Theorem 1]{LPy}); and
\item[2.] $|x^{G}| \geqs q^{m/2}$ for all $x \in G$ of prime order.
\end{enumerate}
Set $c=\lceil 11/\e \rceil$. Then  
$$\what{Q}(G,c) = \sum_{i=1}^{k}{|x_{i}^{G}|\cdot {\rm fpr}(x_{i})^{c}}<\sum_{i=1}^{k}{|x_{i}^{G}|^{-10}} \leqs k\cdot(q^{m/2})^{-10}\leqs q^{-m}$$
and thus $Q(G,c)$ tends to $0$ as $|G|$ tends to infinity, as required.
\end{proof}

\subsection{Further results}

As the above sketch proof indicates, the constant $c$ in Theorem \ref{t:bls} depends on the constant $\e$ in \eqref{e:starr}, and is therefore undetermined. However, by applying the stronger fixed point ratio estimates in \cite{Bur_1, Bur_2, Bur_3, Bur_4} and \cite{LLS2}, it is possible to show that $c=6$ is optimal. Indeed, the following result, which is proved in the sequence of papers 
\cite{Bur7,BGS,BLS,BOB}, reveals a striking dichotomy for almost simple primitive groups: either the base size can be arbitrarily large (standard groups), or there exists an extremely small base (non-standard groups). This solves Problem B as stated in the introduction.

\begin{thm}\label{main}
Let $G \leqs {\rm Sym}(\O)$ be a non-standard permutation group. Then $b(G) \leqs 7$, with equality if and only if $G={\rm M}_{24}$ in its natural action on $24$ points. Moreover, the probability that a random $6$-tuple in $\O$ forms a base for $G$ tends to $1$ as $|G|$ tends to infinity.
\end{thm}

\begin{ex}
To illustrate the proof of Theorem \ref{main} for exceptional groups, let us briefly sketch an argument to show that $b(G) \leqs 5$ for $G = E_8(q)$. 

First assume $H$ is small, say $|H| \leqs q^{88}$, and write $\what{Q}(G,5)=\a+\b$ where $\a$ is the contribution to the summation from elements $x \in G$ with $|x^G| \leqs \frac{1}{2}q^{112}$. Observe that 
$$\b<\frac{1}{2}q^{112}\left(\frac{2|H|}{q^{112}}\right)^5 \leqs \frac{1}{2}q^{112}\left(\frac{2q^{88}}{q^{112}}\right)^5 = 16q^{-8}.$$
Suppose $x \in G$ has prime order and $|x^G|\leqs \frac{1}{2}q^{112}$. The sizes of the conjugacy classes of elements of prime order in $G$ are  available in the literature and by inspection we see that $x$ is unipotent of type $A_1$ or $2A_1$ (in terms of the standard Bala-Carter labelling of unipotent classes). There are fewer than $2q^{92}$ such elements in $G$, and \cite[Theorem 2]{LLS2} implies that ${\rm fpr}(x) \leqs 2q^{-24}$, hence
$$\a<2q^{92}(2q^{-24})^5 = 64q^{-28}$$
and the result follows.

To complete the analysis, we may assume $|H|>q^{88}$. As discussed in Section \ref{ss:ex}, if $H$ is a maximal parabolic subgroup then it is possible to compute very accurate fixed point ratio estimates using character-theoretic methods (recall that for unipotent elements, this relies on L\"{u}beck's work on Green functions for exceptional groups). This allows us to obtain strong upper bounds on $b(G)$ for parabolic actions. Moreover, when combined with the trivial lower bound $b(G) \geqs \log |G| / \log |\O|$, we get the exact base size in all but one case (here $P_i$ denotes the maximal parabolic subgroup of $G$ corresponding to the $i$-th node in the Dynkin diagram of $G$ with respect to the standard Bourbaki \cite{Bou} labelling):
$$\begin{array}{c|cccccccc}
H & P_1 & P_2 & P_3 & P_4 & P_5 & P_6 & P_7 & P_8 \\ \hline
b(G) & 4 & 3 & 3 & 3 & 3 & 3 & \mbox{$3$ or $4$} & 5 
\end{array}$$
Note that $\log |G|/\log |\O| = 3-o(1)$ when $H=P_7$, so we expect $b(G)=4$ is the correct answer in this case.

Finally, let us assume $|H|>q^{88}$ and $H$ is non-parabolic. By applying a fundamental subgroup structure theorem of Liebeck and Seitz (see \cite[Theorem 8]{LS03}), we deduce  that $H$ is of type $D_8(q)$, $E_7(q)A_1(q)$ or $E_8(q_0)$ with $q=q_0^2$. Let us assume $H$ is of type $D_8(q)$; the other cases are similar. Let $K$ be the algebraic closure of $\mathbb{F}_{q}$ and set $\bar{G}=E_8(K)$ and $\bar{H}=D_8(K)$, so we may view $G$ as $\bar{G}_{\s}$, and similarly $H$ as $\bar{H}_{\s}$, for a suitable Frobenius morphism $\s$ of $\bar{G}$. Write $\what{Q}(G,5) = \a+\b$, where $\a$ is the contribution from semisimple elements. 

Let $x \in G$ be a semisimple element of prime order. By applying \cite[Lemma 4.5]{LLS2}, we have
$${\rm fpr}(x) \leqs \frac{|W(\bar{G}){:}W(\bar{H})|\cdot 2(q+1)^8}{q^{\delta(x)}(q-1)^8} = 270\left(\frac{q+1}{q-1}\right)^8q^{-\delta(x)},$$
where $W(\bar{X})$ is the Weyl group of $\bar{X}$ and $\delta(x) = \dim x^{\bar{G}} - \dim(x^{\bar{G}}\cap \bar{H})$. If $C_{\bar{G}}(x)^0$ is not of type $D_8$, $E_7T_1$ nor $E_7A_1$, then \cite[Theorem 2]{LLS} implies that $\delta(x) \geqs 80$ and we deduce that ${\rm fpr}(x) <q^{-59}$. Therefore, if $\a_1$ denotes the contribution to $\a$ from these elements then 
$$\a_1<|G|\cdot (q^{-59})^5 < q^{248}\cdot q^{-295} = q^{-47}.$$
There are fewer than $q^{130}$ remaining semisimple elements in $G$ and by applying \cite[Theorem 2]{LLS2} we deduce that their contribution is less than $q^{130}\cdot (q^{-37})^5 = q^{-55}$. In particular, 
$$\a<q^{-47}+q^{-55}$$ 
is tiny.

For unipotent elements there is a distinction between the cases $q$ even and $q$ odd (recall that we are only interested in elements of prime order). In any case, the fusion in $\bar{G}$ of unipotent classes in $\bar{H}$ has been determined by Lawther \cite{Law} and using these results it is straightforward to show that $\b$ is also small (and tends to zero as $q$ tends to infinity). See the proof of \cite[Lemma 4.5]{BLS} for the details.
\end{ex}

For classical groups, Theorem \ref{t:bur} is the key ingredient in the proof of Theorem \ref{main}, which roughly states that $\e \sim 1/2$ is optimal in \eqref{e:starr}. In order to use it, let $m$ be the dimension of the natural module for $G_0$ and set  
$$\eta_{G}(t)=\sum_{i=1}^{k}{|x_{i}^{G}|^{-t}}$$
for $t \in \mathbb{R}$, where $x_1,\ldots,x_k$ represent the $G$-classes of elements of prime order in $G$. If $m \geqs 6$, then careful calculation reveals that $\eta_{G}(1/3) < 1$. Therefore, by combining this with the generic upper bound ${\rm fpr}(x)<|x^G|^{-1/2+1/m}$ from Theorem \ref{t:bur}, we deduce that
$$\what{Q}(G,4) <\sum_{i=1}^{k}{|x_{i}^{G}|^{1+4(-\frac{1}{2}+\frac{1}{m})}} \leqs \eta_{G}(1/3) < 1$$
if $m \geqs 6$, and thus $b(G) \leqs 4$. In this way, we can establish the following sharpened version of Theorem \ref{main} for classical groups (see \cite[Theorem 1]{Bur7}).

\begin{thm}\label{bthm2}
Let $G$ be a non-standard classical group with point stabiliser $H$. Then $b(G) \leqs 5$, with equality if and only if $G={\rm PSU}_{6}(2).2$ and $H={\rm PSU}_{4}(3).2^{2}$. Moreover, the probability that a random $4$-tuple in $\O$ forms a base for $G$ tends to $1$ as $|G|$ tends to infinity.
\end{thm}

The problem of determining the precise base size of every non-standard group is an ongoing project of the author with Guralnick and Saxl. We finish by reporting on some recent work towards this goal.  

\vs

\subsubsection{Alternating and sporadic groups}

If $G_0$ is a (non-standard) alternating or sporadic group, then $b(G)$ has been calculated in all cases. For example, if $G_0={\rm Alt}(n)$ then \cite[Theorem 1.1]{BGS} implies that $b(G)=2$ if $n>12$ (if $G={\rm Alt}(12)$ and $H = {\rm M}_{12}$, then $b(G)=3$). Similarly, if  $G=\mathbb{M}$ is the Monster sporadic group, then $b(G)=2$ unless $H = 2.\mathbb{B}$, in which case $b(G)=3$ (see \cite{BOB}).

To handle the symmetric and alternating groups, we first observe that $H$ is a primitive subgroup (this follows from the non-standard hypothesis), so we can use a well known result of Mar\'{o}ti \cite{Maroti} to bound $|H|$ from above. This is combined with Theorem \ref{t:gm} on the minimal degree of $H$, which tells us that either $\mu(H) \geqs n/2$, or $H$ is a product-type group arising from the action of a symmetric group on $k$-sets. The latter situation can be handled directly, whereas in the general case we translate the bound on $\mu(H)$ into a lower bound on $|x^G|$ for all $x \in H$ of prime order. This is useful because 
$$\what{Q}(G,2)<|H|^2\max_{1\neq x \in H}|x^G|^{-1}.$$

The proof for sporadic groups relies heavily on computational methods, together with detailed information on their conjugacy classes, irreducible characters and subgroup structure that is available in \textsf{GAP} \cite{GAP4}. Further work is needed to handle the Baby Monster and the Monster (see \cite{BOB, NNOW} for more details).

As an aside, it is worth noting that determining the exact base size for the \emph{standard} groups with an alternating socle is a difficult combinatorial problem. Indeed, this is an open problem, even for the action of ${\rm Sym}(n)$ on $k$-sets. See \cite{Halasi} for the best known results in this particular case.

\vs

\subsubsection{Classical groups}

Suppose $G_0$ is a classical group, so $H$ is either geometric or non-geometric. In \cite{BGS2}, probabilistic methods are used to determine the precise base size of all non-geometric actions of classical groups. Here the key ingredient is Theorem \ref{t:gl}, combined with a detailed analysis of the low-dimensional irreducible representations of quasisimple groups. As discussed in Section \ref{sss:s}, the lower bound on $\nu(x)$ in part (ii) of Theorem \ref{t:gl} yields a lower bound on $|x^G|$, so our approach is somewhat similar to the one we used for symmetric and alternating groups (although the details are more complicated in this situation). 

The following result is a simplified version of \cite[Theorem 1]{BGS2}.

\begin{thm}
Let $G \leqs {\rm Sym}(\O)$ be a non-standard classical group with socle $G_0$ and point stabiliser $H \in \mathcal{S}$. Assume $n>8$, where $n$ is  the dimension of the natural module for $G_0$. Then one of the following holds:
\begin{itemize}\addtolength{\itemsep}{0.2\baselineskip}
\item[{\rm (i)}] $b(G)=2$;
\item[{\rm (ii)}] $b(G)=3$ and $(G,H) = ({\rm O}_{14}^{+}(2), {\rm Sym}(16))$, 
$({\rm O}_{12}^{-}(2), {\rm Sym}(13))$, \\
$(\O_{12}^{-}(2), {\rm Alt}(13))$ or $(\O_{10}^{-}(2), {\rm Alt}(12))$;
\item[{\rm (iii)}] $b(G)=4$ and $(G,H) = ({\rm O}_{10}^{-}(2), {\rm Sym}(12))$.
\end{itemize}
\end{thm}

The analysis of the geometric actions of classical groups is more difficult, and so far we only have partial results. For example, we can show that $b(G)=2$ if $H$ is a subfield subgroup corresponding to a subfield of $\mathbb{F}_{q}$ of odd prime index. The analysis in some cases is rather delicate; for example, it can be difficult to decide if $b(G)=2$ or $3$. This sort of situation tends to arise when $|H| \sim |G|^{1/2}$, which often corresponds to a case where $H$ is the centraliser in $G$ of an involution (at least when $q$ is odd). We anticipate that our recent work in \cite{BGS3} on base sizes for primitive actions of simple algebraic groups will play a role in this analysis.

\vs

\subsubsection{Non-standard groups with large base size}

The proof of Theorem \ref{main} reveals that there are infinitely many exceptional groups with $b(G) \geqs 5$. In fact, it has recently been shown  that there are infinitely many with $b(G)=6$ (see \cite[Theorem 11]{BGS3}) and with some additional work it is possible to determine them all (see \cite{Bur_base6}).

\begin{thm}
Let $G \leqs {\rm Sym}(\O)$ be a non-standard permutation group with socle $G_0$ and point stabiliser $H$. Then $b(G)=6$ if and only if one of the following holds:
\begin{itemize}\addtolength{\itemsep}{0.2\baselineskip}
\item[{\rm (i)}] $(G,H) = ({\rm M}_{23}, {\rm M}_{22})$, $({\rm Co}_{3}, {\rm McL}.2)$, $({\rm Co}_{2}, {\rm PSU}_{6}(2).2)$, \\
or $({\rm Fi}_{22}.2, 2.{\rm PSU}_{6}(2).2)$;
\item[{\rm (ii)}] $G_0 = E_7(q)$ and $H = P_7$; 
\item[{\rm (iii)}] $G_0 = E_6(q)$ and $H = P_1$ or $P_6$.
\end{itemize}
\end{thm}

In cases (ii) and (iii), the usual estimates via fixed point ratios yield $b(G) \in \{5,6\}$, so more work is needed to pin down the precise answer. To do this we apply some recent results from \cite{BGS3} on bases for simple algebraic groups. 

For example, consider case (ii). Let $\bar{G}=E_7$ and $\bar{H}$ be the corresponding algebraic groups over $\bar{\mathbb{F}}_{q}$ (so $\bar{H}$ is a maximal parabolic subgroup of $\bar{G}$ with Levi factor of type $E_6$) and let $\s$ be a Frobenius morphism of $\bar{G}$ such that $(\bar{G}_{\s})'=G_0$. We may assume that $\bar{H}$ is $\s$-stable. In \cite[Section 5]{BGS3} we show that the generic $5$-point stabiliser in $\bar{G}$ with respect to the action on the coset variety $\bar{G}/\bar{H}$ is $8$-dimensional (here \emph{generic} means that there is a non-empty open subvariety $U$ of $(\bar{G}/\bar{H})^5$ such that the stabiliser in $\bar{G}$ of any tuple in $U$ is $8$-dimensional). By considering the fixed points under $\s$, we deduce that every $5$-point stabiliser in the finite group  $G$ is nontrivial, for any $q$, and the result follows. Note that the connected component of the generic $5$-point stabiliser is not a torus (because it is $8$-dimensional), so there are no complications with split tori when $q=2$.

\end{document}